\documentclass[11pt]{amsart}

\usepackage{enumerate,url,amssymb,  nicefrac, mathrsfs, graphicx, pdfsync}
\newtheorem{theorem}{Theorem}[section]
\newtheorem{lemma}[theorem]{Lemma}
\newtheorem*{lemma*}{Lemma}

\newtheorem{proposition}[theorem]{Proposition}
\newtheorem{corollary}[theorem]{Corollary}

\theoremstyle{definition}
\newtheorem{definition}[theorem]{Definition}

\newtheorem{question}[theorem]{Question}

\theoremstyle{remark}
\newtheorem{remark}[theorem]{Remark}

\numberwithin{equation}{section}

\newcommand{\abs}[1]{\lvert#1\rvert}

\newcommand{\C}{\mathbb{C}}

\newcommand{\W}{\mathscr{W}}

\newcommand{\onto}{\xrightarrow[]{{}_{\!\!\textnormal{onto\,\,}\!\!}}}
\newcommand{\into}{\xrightarrow[]{{}_{\!\!\textnormal{into\,\,}\!\!}}}
\newcommand{\bydef}{\stackrel {\textnormal{def}}{=\!\!=} }

\DeclareMathOperator{\re}{Re}
\DeclareMathOperator{\im}{Im}

\begin{document}

\title{The Dirichlet Principle for Inner Variations}

\author[T. Iwaniec]{ Tadeusz Iwaniec}
\address{Department of Mathematics, Syracuse University, Syracuse,
NY 13244, USA}
\email{tiwaniec@syr.edu}

\author[J. Onninen]{Jani Onninen}
\address{Department of Mathematics, Syracuse University, Syracuse,
NY 13244, USA and Department of Mathematics and Statistics, P.O.Box 35 (MaD) FI-40014 University of Jyv\"askyl\"a, Finland}
\email{jkonnine@syr.edu}
\thanks{ T. Iwaniec was supported by the NSF grant DMS-1802107.
J. Onninen was supported by the NSF grant DMS-1700274.}

%    General info

%\date{\today}


\maketitle

\begin{abstract}
 We are  concerned with the Dirichlet energy of mappings defined on domains in the complex plane. The motivation behind our questions, however, comes from more general energy integrals of mathematical models of {\it Hyperelasticity}. The  Dirichlet Principle, the name coined by Riemann,  tells us that the \textit{outer variation} of a harmonic mapping increases its energy. Surprisingly, when one jumps into details about \textit{inner variations}, which are just a change  of independent variables, new equations and related questions start to matter. The inner variational equation, called the \textit{Hopf Laplace equation},  is no longer the Laplace equation. Its  solutions are generally not harmonic; we refer to them as \textit{Hopf harmonics}. The natural question that arises is how does a change of variables in the domain of a Hopf harmonic map affect its energy? We show, among other results,  that in case of a simply connected domain the energy increases. This should be viewed as Riemann's Dirichlet Principle for Hopf harmonics.

 The Dirichlet Principle for Hopf harmonics in domains of higher connectivity is not completely solved.   What complicates the matter is the insufficient knowledge of  global structure of trajectories of the associated Hopf quadratic differentials, mainly because of the presence of recurrent trajectories.  Nevertheless, we have established the Dirichlet Principle whenever the Hopf differential admits closed trajectories and crosscuts. Regardless of these assumptions, we established
 the so-called \textit{Infinitesimal Dirichlet Principle} for  all domains and all Hopf harmonics. Precisely, the second order term of inner variation of a Hopf harmonic map is always nonnegative.

 The topics presented in this paper open new directions toward mathematical foundations of Hyperelasticity. In particular, the use of quadratic differentials in the context of hyperelasticity should appeal to both mathematical analysts and researchers in the engineering fields.

\end{abstract}

%\tableofcontents

\section{Introduction}

\subsection{Motivation} Before embarking upon the results, let us consider arbitrary bounded domains  $\mathbb X \,$ and $\,\mathbb Y\,$ in $\,\mathbb R^n\,$. We shall actually investigate  in detail only the case $\,n= 2\,$. Although the $\,n\,$-dimensional Riemannian manifolds are not in the center of our investigation, the ideas really crystalize in a diferential-geometric setting. Thus we suggest, as a possibility, to think of $\mathbb X\,$ and $\,\mathbb Y\,$ as Riemannian $\,n\,$-manifolds or surfaces when $\,n=2\,$.
The subject matter is about Sobolev mappings $\, h\colon \mathbb X\,\rightarrow \mathbb Y\,$ of class $\,\mathscr W^{1,p}_{\textnormal{loc}}(\mathbb X, \mathbb R^n)\,,\, 1 \leqslant p \leqslant \infty\,$. The chief part of this paper is highly motivated by the mathematical models of Nonlinear Elasticity (NE) originated in \cite{Antman, Ball1, Ball2, Ciarlet, MarsdenHughes}.  The reference configuration $\,\mathbb X\,$, the deformed configuration $\,\mathbb Y\,$,  and the elastic deformation, usually a homeomorphism $\,h \colon  \mathbb X  \onto \mathbb Y\,$, thus named,  have a well defined linear tangent map $\,Dh : \textnormal T_x\mathbb X  \rightarrow \textnormal T_y\mathbb Y\, ,\, y = h(x)\,$, at almost every point $\,x \in \mathbb X\,$, called a \textit{deformation gradient}. In the Euclidean setting $\,Dh\,$ is just a measurable function on $\,\mathbb X\,$ whose values are $\, n \times n\,$-matrices, so we write $\,Dh(x)  \in \mathbb R^{n\times n}\,$. The adjoint differential $\,D^*h(x) : \textnormal T_y\mathbb Y  \rightarrow \textnormal T_x\mathbb X\,$, represented by the transpose matrix  of $Dh(x)$, gives rise to the Hilbert-Schmidt norm $\, |Dh| \bydef \sqrt{\textnormal{Tr} (D^*h\cdot Dh)} = \sqrt{\langle Dh | Dh\rangle} \,$.

The theory of hyperelasticity is concerned with the  \textit{stored energy}, usually defined for Sobolev homeomorphisms $\,h : \mathbb X \onto \mathbb Y\,$ and their weak limits:
\begin{equation}\label{EnergyIntegral}
\mathscr E[h]  =  \int_\mathbb X \textbf{E}(x,h, Dh) \, \textbf{d} x \; < \infty\,,
\end{equation}
for the purpose of determining its infimum.  The major player is the Jacobian determinant $\,J_h(x) = J(x,h) = \textnormal{det}\,Dh(x)\,$ which is often assumed to be nonnegative in order to comply with so-called  \textit{Principle of Non-Interpenetration of Matter}~\cite{Ball2, Ball3, BPO1, Ciarlet,  CL, IOhy, IOinv}. Accordingly, it is energetically impossible to compress part of the hyperelastic body to zero volume; the Jacobian must be positive.

It is a persistent misconception that the energy-minimal homeomorphisms  must satisfy the Lagrange-Euler equation. Whereas, upon a little reflection on the \textbf{outer variation}
$$\,h_\varepsilon(x) \bydef  h(x) + \varepsilon \,\eta(x) \;, \; \textnormal{with} \;\,\eta \in \mathscr C^\infty_0 (\mathbb X, \mathbb R^n), \,$$  such a view becomes well out of reality. The variations $\,h_\varepsilon\,$ are generally not  homeomorphisms of $\,\mathbb X \onto\,\mathbb Y\,$ and, even more,  the Jacobian may change sign. This being so, one quickly runs into
serious difficulty when trying to apply the \textit{Direct Method in the Calculus of Variations} by passing to a weak limit of an
energy-minimizing sequence of Sobolev homeomorphisms; injectivity is lost. That is why, one must accept limits of homeomorphisms as legitimate hyperelastic deformations~\cite{IOmono, IOws}. Besides these concerns, even if such a limit possesses the least energy it is not generally possible to write down a Lagrange-Euler equation for the minimal mapping. An immediate example is the \textit{Neo-Hookean energy}:
\begin{equation}
\mathsf E_q^p[h] = \int_\mathbb X \left (| Dh(x) |^p  \,+\; \left [J_h(x)\right]^{-q} \right )\; \textnormal d x\;,\;\; 1< p < \infty\;,\; q >0 .
\end{equation}
which does not authorize to use outer variations. But it  allows for the inner variations.
\begin{definition}
 By the (total) \textit{inner variation}  of $\,h : \mathbb X \to \mathbb R^n\,$ we mean a family of mapping $\,h_\phi: \mathbb X \into \mathbb R^n\,$,  $\,h_\phi(x) \bydef h(\phi(x))\,$,  in which $\,\phi : \overline{\mathbb X }\onto \overline{\mathbb X}\,$ are $\,\mathscr C^\infty\,$-diffeomorphisms, referred sometimes as change of variables in $\, \mathbb X\,$.
 \end{definition}

One of the reasons why the inner variations are  advantageous over outer variations is that $\,h_\phi(\mathbb X) = h(\mathbb X)\,$.
 Although in this most general setting we do not prescribe the boundary values of $\,h\,$, its boundary behavior is still involved via the assumption $\, h(\mathbb X) = \mathbb Y\,$. In nonlinear elasticity \cite{Ball1, Ball3, Ball4, Ciarlet, CN} this is called  \textit{frictionless problem} as it allows for ``tangential slipping" along the boundary.  One can realize it physically by deforming an incompressible material confined in a box.
 In the \textit{Geometric Function Theory} (GFT)~\cite{AIMb, CIKO, Ge, GMPb, HKb, IMb, IOne, IOan, IOlatem}, on the other hand, the frictionless deformations naturally occur  in generalizing Riemann's Mapping Theorem,  where prescribing the boundary values of $\,h\,$ is an ill posted problem.

 Minimization of the energy (\ref{EnergyIntegral}), subject to frictionless deformations,  leads to a \textit{variational equation} on $\,\mathbb X\,$ and additional equations on $\,\partial \mathbb X\,$, see e.g.~\cite{IKOLip, IOan}. In order to cover the boundary value problems as well, we shall confine ourselves to diffeomorphisms $\,\phi : \overline{\mathbb X }\onto \overline{\mathbb X}\,$ that are equal to the identity map on $\,\partial \mathbb X\,$. It will simplify the arguments and cause no loss of generality to assume that $\,\phi(x) \equiv x\,$ near $\,\partial \mathbb X\,$.
  Thus, we choose and fix a test function $\,\eta \in \mathscr C^\infty_0(\mathbb X, \mathbb R^n)\,$. For all sufficiently small $\varepsilon \in \mathbb R\,$ the mappings $\,  \phi(x) \bydef  x + \varepsilon \eta(x)\,$ are diffeomorphisms of $\,\overline{\mathbb X}\,$ onto itself.
 \begin{definition} \label{DefinitionInnerVariation}The (internal or local) inner variation of $\,h \,$  is defined by
  \begin{equation}\label{NotationInnerVariation}\,h^\varepsilon(x)  \bydef h(x + \varepsilon \eta(x))\,,\; \textnormal{where}\;\eta\in \mathscr C^\infty_0(\Omega)\;\;\textnormal{and} \;\varepsilon \in \mathbb R .
  \end{equation}
  Here the parameter $\,\varepsilon\,$ is small enough to ensure the Jacobian condition:
  \begin{equation}\label{JacobianCondition}\, \det [ I \,+\, \varepsilon\, D\eta ] \,> 0\,,\; \textnormal{everywhere in}\,\,\Omega .
  \end{equation}
  \end{definition}
Clearly,  if $\,h\,$ is an energy-minimal deformation among all inner variations, then it satisfies the  so-called \textit{inner variational equation}:
 \begin{equation}\label{InnerEquation}
 \frac{\textnormal d}{\textnormal {d} \varepsilon}\,\mathscr E[h^\varepsilon] \,\Big |_{\varepsilon = 0\,} \, =  \,0 \,, \,\; \textnormal{for all} \;\,\eta \in \mathscr C^\infty_0(\mathbb X, \mathbb R^n) .
 \end{equation}

It is generally a highly nontrivial question whether the converse holds; and this is our primary question that we address in this paper.
\vskip0.2cm
\begin{tabular}{|p{11.3cm}|}
\hline \vskip0.1cm
\begin{question}[General Dirichlet Principle]\label{GeneralDirichletPrinciple}
Suppose that a mapping $\,h :\mathbb X \onto \mathbb Y\,$ of finite energy  at \eqref{EnergyIntegral} solves the equation (\ref{InnerEquation}). Does every inner variation of $\,h\,$ increase its energy? Precisely, is it true that $\,\mathscr E[h] \leqslant \mathscr E[h^\varepsilon]$?
\end{question}
$\;$
  \\\hline
\end{tabular}\\

Inner-variational equations are also known as  \textit{energy-momentum} or \textit{equilibrium} equations, etc~\cite{Cob, SSe, Ta}. In recent studies there has been an intense exploration of the inner variations. Applications are plentiful and quite significant. For example, in the study of the regularity of energy-minimal mappings the unavailability of the Lagrange-Euler equation is a major source of difficulties. Such a difficulty is well recognized in the theory of nonlinear elasticity~\cite{Ba13, Ba12, SS}. In different circumstances,  a deeper understanding of the \textit{Hopf-Laplace equation}, see formula (\ref{Hopf-Laplace-Equation}) below, helped us to gain Lipschitz regularity of solutions (not necessarily energy-minimal) of a wide class of conformally invariant equations\,\cite{IKOLip}.\\
 Question \ref{GeneralDirichletPrinciple},  as posed in such a generality,  seems to be over-committed at the current stage of developments. That is why in this paper we  undertake a detailed study of the Dirichlet energy in the planar domains. The use of complex methods (quadratic differentials in particular) are encouraging enough to merit such investigation.

\subsection{Planar Dirichlet Energy}
 From now on  $\,h \colon \Omega \rightarrow \mathbb C\,$ is a Sobolev mapping of class  $\,\mathscr W^{1,2}(\Omega)\,$ defined on a domain $\,\Omega \subset \mathbb C\,$ in the complex plane $\,\mathbb C = \{\, z = x + i y \colon  x,y \in \,\mathbb R\}\,$, which we dress   with d'Alambert's complex derivatives.
$$
 h_z  = \frac{\partial h}{\partial z} = \frac{1}{2} \left( \frac{\partial}{\partial x} - i \frac{\partial}{\partial y} \right)h \;\;\;\textnormal {and}\;\;\;
h_{\bar z} = \frac{\partial h}{\partial \bar z}  = \frac{1}{2} \left( \frac{\partial}{\partial x} + i \frac{\partial}{\partial y} \right)h \, .
$$
In this notation the Dirichlet energy takes the form:
$$
\mathscr E[h] \; \bydef  \,\frac{1}{2}\,\int_\Omega \abs{Dh(x,y)}^2\, \textnormal dx\, \textnormal d y   \,=\,  \int_\Omega \left (|\,h_z(z) \,|^2 \;+\, |h_{\bar z} (z)\,|^2\,\right) \, \textbf{d} z \, .
$$
 Hereafter $\,\textbf{d} z\,$ stands for the area element in $\,\mathbb C\,$, $\,\textbf{d} z = \textnormal dx\, \textnormal d y \, =  \frac{i}{2} \,\textnormal d z \wedge \textnormal d\bar{z}$.

\subsection{Dirichlet Principle}
  Historically, the existence of the energy-minimal solutions was hinged on physical interpretations. This was taken for granted (until Karl Weierstrass' constructed a counter-example) by numerous eminent mathematicians, including Bernhard Riemann who actually coined the term \textit{Dirichlet's Principle}.  Let us encapsulate this principle as:
  \vskip0.2cm
  \begin{tabular}{|p{11.3cm}|}
\hline \vskip0.1cm
\begin{center} \textbf{Riemann's Dirichlet Principle} \end{center}
  A function $\,h\in \mathscr W^{1,2}(\Omega) \,$ solves the Laplace equation
  \begin{equation} \label{LaplaceEquation}
  h_{z\bar{z}} \;= \,\frac{\partial h_z}{\partial \bar z}\,  \equiv 0\,\;\;\,\;\; \textnormal{(in the sense of distributions)}
  \end{equation}
  if and only if  its \textit{outer variations} increase the energy.  \\
$\;$
  \\\hline
\end{tabular}\\

\subsubsection{Outer Variation}

 Recall that the term outer variation of $\, h : \Omega   \rightarrow \mathbb C\,$ refers to a one parameter family $\,\{h^\varepsilon\,\}_{\varepsilon \in \mathbb R}\,$ of mappings $\,h^\varepsilon : \Omega   \rightarrow \mathbb C\,$  defined by the rule:
 \begin{equation}\label{OuterVariation}
 h^\varepsilon (z) = h(z) \,+ \,\varepsilon \,\eta(z)\, ,\;\;\textnormal {where}\;\;\eta \in \mathcal C^\infty_0 (\Omega).
 \end{equation}
 The energy of $\,h^\varepsilon\,$ is a quadratic polynomial in $\,\varepsilon$.
\begin{equation}\label{OuterVariationOfEnergies}
\mathscr E[h^\varepsilon] \; = \; \mathscr E[h]  - 4 \varepsilon \re \int\; \overline{\eta}\, \frac{\partial  h_z}{\partial \bar z}\, \textbf d z\,  \, \; +\; \varepsilon^2 \int \left ( | \eta_z |^2 \,+ | \eta_{\bar z} |^2  \right )\,\textbf d z \, .
\end{equation}
The Dirichlet Principle is now readily inferred from the first order power term by letting $\,\varepsilon\,$ go to zero. Since the test functions $\,\eta\,$ assume complex values, we conclude that  $\,h\,$ is harmonic if and only if

\begin{itemize}
\item[(A)]\quad\quad\quad$\,\frac{\textnormal d}{\textnormal {d} \varepsilon}\,\mathscr E[h^\varepsilon] \,\Big |_{\varepsilon = 0\,} \, =  \,0 \,$, \quad \textnormal {for all$\,\;\eta \in\mathscr C^\infty_0(\Omega)\,$}.
\end{itemize}

Now, having this equation, the second order power term (named \textit{second outer variation}) turns out to be nonnegative,
\begin{itemize}
\item [(B)] \quad\quad\quad $\,\frac{\textnormal d^2}{\textnormal {d} \varepsilon^2}\,\mathscr E[h^\varepsilon] \,\Big |_{\varepsilon = 0\,}\; \geqslant  \;0$\, , \quad for all $\,\eta \in\mathscr C^\infty_0(\Omega)\,$
    \end{itemize}

This inequality actually holds for every parameter $\,\varepsilon\,$, so we have
\begin{itemize}
    \item [(C)]\quad\quad\quad $\mathscr E[h] \, \leqslant \mathscr E[h + \eta]\,\,, \quad \textnormal{for all}\, \,\eta \in \mathscr C^\infty_0 (\Omega)\,.$
\end{itemize}
The equality occurs if and only if $\,\eta \equiv 0\,$.

\vskip0.1cm
\subsubsection{Inner Variation}

Let  $\,h \colon \Omega \to \C\,$ be a mapping of Sobolev class $\W^{1,2} (\Omega)\,$ in a domain $\,\Omega \subset \C\,$. Where it is important to distinguish different meanings of $\,\Omega\,$, one as the domain of definition of the variation and the other as its image,  we designate different letters $\,z\,$ and  $\xi\,$ for the notation of the variables in $\,\Omega\,$. Accordingly,  $\,h= h(\xi)\,$, where   $\,\xi\,$ can also be viewed as $\,\mathscr C^\infty\,$-smooth diffeomorphism of $\,\xi : \Omega \onto \Omega\,$. Precisely $\, \xi(z) = z + \psi(z)\,$, where $\, \psi \in \mathscr C_\circ^\infty (\Omega)$ satisfies the positive Jacobian condition:
\begin{equation} \label{admissibleDeformations}
J_\xi (z)= \abs{\xi_z(z)}^2 -  \abs{\xi_{\bar z}(z)}^2 >0 \qquad \textnormal{for all } z \in \Omega \, .
\end{equation}

Implicit Function Theorem and topological degree arguments combined reveal that $\,\xi : \Omega \onto \Omega\,$ has an inverse, also denoted by $\, z \colon \Omega \onto \Omega\,$, thus  $\,z = z(\xi)\,$.  Both diffeomorphisms $\,\xi \colon \Omega \onto \Omega\,$ and $\,z \colon \Omega \onto \Omega\,$ are understood as change of variables in $\,\Omega\,$.
\begin{definition}[The Total Inner Variation]\label{InnerVariationDefinition} Recall that the term total inner variation of a function $\,h : \Omega \to \mathbb C\,$ refers to any function $\,H = H(z)\,$ defined by the rule:
\begin{equation}\label{InnerVariation}
H(z) = h(\xi(z)) \,\,,\;\; \textnormal{for}\; \,z \in \Omega\,
\end{equation}
 where $\;\xi = \xi(z)\,$ is any diffeomorphism of $\,\Omega\,$ onto itself.
 \end{definition}
In Section \ref{Section3} we inaugurate the following general formula:
\begin{equation}\label{ExpansionFormula}
\begin{split}
&\mathscr E [H] \,-\, \mathscr E [h]= \\&  \int_\Omega \left (|\,H_z(z) \,|^2 \,+\, |H_{\bar z} (z)\,|^2\,\right)  \textbf{d} z
 \;- \int_\Omega \left (|\,h_\xi(\xi) \,|^2 \,+\, |h_{\bar \xi} (\xi)\,|^2\,\right)  \textbf{d} \xi\\& = 2 \int_\Omega \left (|\,h_\xi(\xi) \,|^2 \,+\, |h_{\bar \xi} (\xi)\,|^2\,\right) \frac{|z_{\bar{\xi}}|^2}{|z_\xi|^2 - |z_{\bar{\xi}}|^2}\, \textbf{d} \xi\\&
  -4  \re \int_\Omega h_\xi \overline{h_{\bar \xi}} \;\; \frac{z_{\bar{\xi}}\,\, \overline{z_\xi}}{|z_\xi|^2 - |z_{\bar{\xi}}|^2}\, \textbf{d} \xi
\end{split}
\end{equation}
Hereafter, the differential expression
\begin{equation}
\mathcal  H = \mathcal H(\xi) \;=\; h_\xi \overline{h_{\bar \xi}}\;\;\;\; \textnormal{is called  \textit{Hopf product.}}
\end{equation}
This name is given in recognition of Heinz Hopf's work, see his book \cite[Ch. VI]{HopfBook}.% and Chapter VI "\textit{Simple Closed Surfaces of Genus 0  with Constant Mean Curvature — Generalizations}".

It is immediate from (\ref{ExpansionFormula}) that:
\begin{corollary}[The borderline Case]\label{theSingularCase}
If $\,\mathcal H \equiv 0\,$ almost everywhere, then no inner variation of $\,h\,$ decreases its energy; in symbols,\;$\,\mathscr E[h] \leqslant \mathscr E[H]\,$.
\end{corollary}
 For the equality  and for further discussion of this case  see Section \ref{SingularCase}.
\subsection{First and Second Order Terms of the Inner Variations}
 Choose and fix  an arbitrary complex valued function $\,\eta = \eta(\xi)\,$ of class $\,\mathscr C^\infty_0(\Omega)\,$. For sufficiently small $\,\varepsilon \in \mathbb R\,$ the mapping $\, z = z(\xi) = \xi + \varepsilon \eta(\xi)\,$ represents a diffeomorphic change of variables in $\, \Omega\,$.
 \begin{definition}[The Range of $\,\varepsilon\,$]\label{RangeOfEpsilon}
 Given $\,\eta = \eta(\xi)\,\in \mathscr C^\infty_0(\Omega)\,$, the largest positive number $\, \varepsilon_{\textnormal{max}}\,$ for which
 the mappings $\, z = z(\xi) = \xi \pm \varepsilon \eta(\xi)\,$, with $\,0 < \varepsilon < \varepsilon_{\textnormal{max}}\,$ are diffeomorphisms will hereafter be referred to as the maximal variational parameter. Certainly, $\,\varepsilon_{\textnormal{max}}\,$  depends on the choice of the test function $\,\eta \,\in \mathscr C^\infty_0(\Omega)\,$; for convenience we ignore this dependence.
 \end{definition}
 This just amounts to the inequality
 \begin{equation}
 J_z(\xi)  = | 1\, \pm \,\varepsilon \,\eta_{_\xi} (\xi)\,| ^2 \; -  \varepsilon^2\, | \eta_{_{\bar \xi}} (\xi)\,|^2 \, > 0 \;,\;\; \textnormal{for all}\; \xi \in \Omega
 \end{equation}
 whenever $\,0 < \varepsilon < \varepsilon_{\textnormal{max}}\,$.

 Our ultimate goal is to expand formula (\ref{ExpansionFormula}) in powers of $\,\varepsilon\,$.  Therefore, we consider a one parameter family of inner variations of $h$,  defined for sufficiently small $\,\varepsilon \in \mathbb R\,$ by formula  (\ref{InnerVariation}). Equivalently,
 \begin{equation}
 H_\varepsilon(z)  \bydef  h(\xi) \;, \;\; \textnormal{where}\; z = z(\xi) \bydef \xi + \,\varepsilon \,\eta(\xi)\;\;\textnormal{with}\;\; \xi \in \Omega .
 \end{equation}
 This will bring us to an analogue of formula (\ref{OuterVariationOfEnergies}).

\begin{theorem}[Power Type Expansion] \label{Expansion}  The following expansion  in powers of  $\,\varepsilon\,\approx 0\,$  is in effect.
\begin{equation}\label{ExpansionInEpsilon}
\begin{split}
\mathscr E [H_\varepsilon] &=   \mathscr E [h]\;+ 4 \varepsilon \re \int_\Omega h_\xi \overline{h_{\bar \xi}} \;\eta_{\bar \xi}\;\textnormal{\textbf{d}} \xi\\ &  + 4\varepsilon^2 \left(\frac{1}{2}\int_\Omega \left(   \abs{h_\xi}^2 + \abs{h_{\bar \xi}}^2       \right) \abs{\eta_{\bar \xi}}^2 \,\textbf{d} \xi \,+\, \re \int_\Omega h_\xi \overline{h_{\bar\xi}} \;\;\eta_\xi \eta_{\bar \xi} \;\textnormal{\textbf{d}} \xi\;\right)\\
&+ \;\textnormal{terms with higher powers of}  \;\;\varepsilon\,.
\end{split}
\end{equation}
\end{theorem}
The $\,\varepsilon\,$-term is called the first (inner)  variation of $\,h\,$. This term  vanishes if and only if $\,\re \int_\Omega h_\xi \overline{h_{\bar \xi}} \;\eta_{\bar \xi}\;\textbf{d} \xi = 0\,$, for every test function $\, \eta \in \mathscr C^\infty_0(\Omega)\,$. However, since $\,\eta\,$ is complex -valued this equation also holds when  "$\re\,$" is dropped.
\subsection{Hopf Harmonics} We have the  following equation parallel to (A).
\begin{proposition} The equation

$$(A^{'}) \;\;\quad\quad\quad \,\frac{\textnormal d}{\textnormal {d} \varepsilon}\,\mathscr E[H_\varepsilon] \,\Big |_{\varepsilon = 0\,} \, =  \,0 \, \,\; \textit{holds for all} \;\,\eta \in \mathscr C^\infty_0(\Omega)\,$$ \\
if and only if $\,h\,$ satisfies the so-called \textbf{Hopf-Laplace equation}:
\begin{equation}\label{Hopf-Laplace-Equation}
\frac{\partial}{\partial\bar{\xi}}\, \left (h_\xi \, \overline{h_{\bar \xi}}\right)    \; =\; 0 \; \; \; \textnormal{(\textit{in the sense of distributions})}
\end{equation}
\end{proposition}
In other words,  the \textit{Hopf product}\, $\, \mathcal H(\xi) \bydef  h_\xi \, \overline{h_{\bar \xi}} \in \mathscr L^1(\Omega)\,$ is a holomorphic function in  $\,\Omega\,$.
\begin{definition}[Hopf Harmonics]
The term \textit{Hopf harmonics} refers to $\,\mathscr W^{1,2}_{\textnormal{loc}}(\Omega)\,$-solutions of (\ref{Hopf-Laplace-Equation}).
\end{definition}
\subsection{Infinitesimal Dirichlet Principle}
We shall show that the second order variation; that is, the $\,\varepsilon^2\,$-term in (\ref{ExpansionInEpsilon}) is nonnegative. Thus the condition parallel to (B)  reads as,
\begin{itemize}
\item [($B^{'}$)] \quad\quad\quad\quad\quad $\,\frac{\textnormal d^2}{\textnormal {d} \varepsilon^2}\,\mathscr E[h^\varepsilon] \,\Big |_{\varepsilon = 0\,}\; \geqslant  \;0 \,$\,,\;  for all $\,\eta \in\mathscr C^\infty_0(\Omega)\,$.
    \end{itemize}
Precisely, we shall prove the following:
\begin{theorem} [Infinitesimal Dirichlet Principle] \label{SecOrdInequality}Let $\,h \in \mathscr W^{1,2}_{\textnormal{loc}}(\Omega)\,$ be Hopf harmonic. Then for every $\,\eta \in \mathscr C^\infty_0(\Omega)\,$ it holds
\begin{equation}\label{QuadraticInequality}  \frac{1}{2}\int_\Omega \left(   \abs{h_\xi}^2 + \abs{h_{\bar \xi}}^2       \right) \abs{\eta_{\bar \xi}}^2 \,\textnormal{\textbf{d}} \xi \,\;+\;
 \re \int_\Omega h_\xi \overline{h_{\bar\xi}} \;\;\eta_\xi \eta_{\bar \xi} \;\textnormal{\textbf{d}} \xi\;\geqslant 0
\end{equation}
\end{theorem}
The proof of this theorem needs considerable work, see Sections \ref{Section5}, \ref{RectangularPartition}, \ref{Section7} and \ref{Section8}. \\

There are computational tricks that enable us to prove even more general estimate than (\ref{QuadraticInequality}). Namely, we have

\begin{lemma} \label{InequalityForH} Let $\,\mathcal H \,$ be a holomorphic function in $\,\Omega\,$. Then

\begin{equation} \label{GeneralHolomorphicInequality}
 \int_\Omega   |\mathcal H(\xi)|  \; |\eta_{\bar \xi} |^2 \;  \;\textnormal{\textbf{d}} \xi\; \geqslant \; \left |\,\int_\Omega \mathcal H(\xi)  \;\eta_\xi\, \eta_{\bar \xi}  \;\textnormal{\textbf{d}} \xi\;\right |
\end{equation}
for all test functions $\,0 \not\equiv\eta \in \mathscr C^\infty_0 (\Omega)\,$, see Theorem \ref{CriticalDirection} for an equality.
\end{lemma}

It should be noted that establishing a strict inequality in  (\ref{GeneralHolomorphicInequality}) would imply that
\begin{equation}\label{Stronginequality}
\mathscr E[H_\varepsilon] \,>\, \mathscr E[h]\, , \;  \textnormal{for sufficiently small}\; \varepsilon \not =  0\;
\end{equation}
This case actually arises when $\,J_h \not = 0\,$ a.e. in $\,\Omega\,$, see Theorem \ref{StrictInequality}.\\

Proceeding in this direction  to the higher order variations does not look promising. Instead, we shall explore the length-area method for the Hopf differential $\,\mathcal H(z)\,\textnormal d z\otimes \textnormal d z. $
This will lead us to the exact analogue of Dirichlet Principle for simply connected domains:
\begin{theorem} \label{FullDirichletPrinciple} Let $\,\Omega \subset \mathbb C\,$ be simply connected domain  and $\,h : \Omega \rightarrow \mathbb C\,$ a Hopf harmonic mapping. Then no inner variation of $\,h\,$ decreases its energy. Precisely,
\begin{equation}
\mathscr E[h] \; \leqslant \, \mathscr E[H] \, , \qquad   H(z) = h(\xi(z))
\end{equation}
whenever $\,\xi : \Omega \onto \Omega\,$ is a diffeomorphism equal to the identity near $\,\partial \Omega\,$ .
\end{theorem}

\begin{question}\label{Question}
The question arises whether Theorem  \ref{FullDirichletPrinciple} is still valid for multiple connected domains,  so as to complete Riemann's Dirichlet Priciple for all Hopf harmonics.
\end{question}
Our partial answers to this question are furnished  with a number of examples based on the additional assumptions about trajectories of the Hopf differential $\,\mathcal H(z) \,\textnormal d z \otimes \textnormal d z\,$. Precisely, we shall consider the Strebel type differentials with leminiscate type trajectory structure,  see Theorem \ref{MultiplyConnectedPrinciple}.

\begin{remark}
To make this text available to readers whose knowledge about quadratic differentials may be limited,  we provide definitions and include some computational details when clarity requires it. A standard reference to quadratic differentials is  the book of K. Strebel \cite{Strebel}.
\end{remark}

\section{Outer Variation versus Inner Variation}
By way of illustration consider a  map $\, h(z) = \frac{z}{|z |}\,$ defined in an annulus $\,\Omega \bydef \{ z \colon  0 < r < |z| < 1\;\}\,$. The reader may wish to verify that it  satisfies the Hopf-Laplace equation
$$\;
h_z \overline{h_{\bar z}}  = \frac{-1}{4 z^2}
$$
 Therefore, by~\cite{IOho} and by Theorem \ref{MultiplyConnectedPrinciple} herein,  its (nontrivial)  inner variations increase the energy. On the other hand, there are outer variations which decrease the energy. For, consider the following variation  of $\,h\,$,
 $$
 h(z) + \eta(z) =  \frac{1}{1+r} \left(z + \frac{r}{\bar z} \right)\;,\;\;\textnormal{where}\;\; \eta(z) \bydef \frac{1}{1+r} \left(z + \frac{r}{\bar z} \right)\; -\;\frac{z}{|z |}
 $$
   Note that the function $\,\eta \in \mathscr C^\infty(\Omega) \,$ vanishes on  $\,\partial \Omega\,$.  Since  $\, h + \eta\,$ is a harmonic function with the same boundary values as  $\,h\,$ its energy is smaller than that of $\,h\,$. Of course one may modify slightly $\,\eta\,$ near $\,\partial \Omega\,$ to become a test function of class $\,\mathscr C^\infty_0 (\Omega)\,$. This does not affect the energy of $\, h + \eta\,$ to the extent that it will remain smaller than that of $\,h\,$.

\section{Proof of Formula (\ref{ExpansionFormula}) }\label{Section3}
We start with a derivation of Formula (\ref{ExpansionFormula}). For, recall that $\,H(z)  =  h(\xi(z))\,$. The chain rule yields:
\[
\begin{split}
H_z (z) &= h_\xi \,\xi_z + h_{\bar \xi} \,\overline{\xi_{\bar z}} \\
H_{\bar z} (z) &= h_\xi\, \xi_{\bar z} + h_{\bar \xi}\, \overline{\xi_{z}} \, .
\end{split}
\]
Hence
\[
\begin{split}
\abs{H_z (z)}^2 & = \abs{h_\xi}^2 \abs{\xi_z}^2 + \abs{h_{\bar \xi}}^2 \abs{\xi_{\bar z}}^2 + 2 \re \big(h_\xi \overline{h_{\bar \xi}} \,\, \xi_z \xi_{\bar z}\big) \\
\abs{H_{\bar z} (z)}^2 & = \abs{h_\xi}^2 \abs{\xi_{\bar z}}^2 + \abs{h_{\bar \xi}}^2 \abs{\xi_{z}}^2 + 2 \re \big(h_\xi \overline{h_{\bar \xi}} \,\, \xi_z \xi_{\bar z}\big)
\end{split}
\]
and
\[
\abs{H_z (z)}^2 + \abs{H_{\bar z} (z)}^2  = \left(   \abs{h_\xi}^2 + \abs{h_{\bar \xi}}^2 \right) \left(    \abs{\xi_z}^2+ \abs{\xi_{\bar z}}^2\right) + \,4\,\re\left( h_\xi \overline{h_{\bar \xi}} \,\,\xi_z \xi_{\bar z} \right)
\]

Here both the left and the right hand side are functions in the $\,z\,$-variable. Thus we integrate both sides with respect to the area element \,$\,\textbf{d} z\,$. However, in the integral of the right hand side we make change of variable $\, z = z(\xi)\,$.  The transformation rule of the area element takes the form:  \,
$$\,\textbf{d} z = \frac{\textbf{d} \xi}{ J_\xi(z)} = \frac{\textbf{d} \xi}{ |\xi_z|^2 \,-\, |\,\xi_{\bar{z}}|^2}\,$$
Hence
\[
\begin{split}
& \int_\Omega\abs{H_z (z)}^2 + \abs{H_{\bar z} (z)}^2 \textbf{d} z\, = \\&  \int_\Omega\left(   \abs{h_\xi}^2 + \abs{h_{\bar \xi}}^2 \right) \left( \frac{   \abs{\xi_z}^2+ \abs{\xi_{\bar z}}^2}{|\xi_z|^2 \,-\, |\xi_{\bar{z}}|^2}\right) \textbf{d} \xi + \,4\,\re\int_\Omega\left( h_\xi \overline{h_{\bar \xi}}\right) \left ( \,\frac{\,\xi_z \xi_{\bar z}}{|\xi_z|^2 \,-\, |\xi_{\bar{z}}|^2} \right) \textbf{d} \xi
\end{split}
\]

  We need to express $\,\xi_z(z)\,$ and $\,\xi_{\bar z}(z)\,$ as functions of the $\,\xi\,$-variable. For,  we compute  $\frac{\partial \xi}{\partial z}= \xi_z (z)$ and $\frac{\partial \xi}{\partial \bar z}= \xi_{\bar z} (z)$ by means  of the derivatives of the inverse map $\,z=z(\xi)\,$
$$
\xi_z \,=\, \frac{\overline{z_\xi(\xi)}}{J_z(\xi)}\;,\; \xi_{\bar z} \;=\; \frac{- z_{\bar \xi}(\xi)}{J_z(\xi)}\;\;,\;\; \textnormal{where}\; J_z(\xi) = |z_\xi|^2\,-\,|z_{\bar \xi}|^2  \; > 0
$$
Now, formula (\ref{ExpansionFormula}) is readily inferred  from these equations.
\section{Proof of   Theorem \ref{Expansion}} We take  $\,z = z(\xi) \bydef \xi + \,\varepsilon \,\eta(\xi)\,$ in  (\ref{ExpansionFormula}), where $\,\eta = \eta(\xi)\,$ can be an arbitrary function of class $\,\mathscr C^\infty_0(\Omega)\,$, provided $\,\varepsilon\,$ is small enough. We substitute the derivatives $\,z_\xi = 1 + \,\varepsilon \,\eta_\xi\,$ and $\,z_{\bar \xi} = \varepsilon\,\eta_{\bar \xi} \,$  into  formula (\ref{ExpansionFormula}) to obtain,

\begin{equation}\label{powerExpansion}
\begin{split}
\mathscr E [H] \,-\, \mathscr E [h]&= \\& = 2 \varepsilon ^2 \int_\Omega \left (|\,h_\xi(\xi) \,|^2 \,+\, |h_{\bar \xi} (\xi)\,|^2\,\right) \frac{|\eta_{\bar{\xi}}|^2}{|1 + \varepsilon \eta_\xi|^2 - |\varepsilon\eta_{\bar{\xi}}|^2}\, \textbf{d} \xi\\&
  -4 \varepsilon \re \int_\Omega h_\xi \overline{h_{\bar \xi}} \;\; \frac{\eta_{\bar{\xi}}\,\, \overline{1 +\varepsilon \eta_\xi}}{|1 + \varepsilon \eta_\xi|^2 - |\varepsilon \eta_{\bar{\xi}}|^2}\, \textbf{d} \xi
\end{split}
\end{equation}
Since we are interested only in terms up to order $\,\varepsilon ^2\,$ we only need to take into account the following expansions
$$\,\frac{1}{|1 + \varepsilon \eta_\xi|^2 - |\varepsilon\eta_{\bar{\xi}}|^2}\, \approx 1\,\; \textnormal{(in the first integral)} $$
 $$\,\frac{1}{|1 + \varepsilon \eta_\xi|^2 - |\varepsilon\eta_{\bar{\xi}}|^2}\, \approx \frac{1}{1 + 2 \varepsilon \re \eta_\xi} \, \approx 1\,- 2\varepsilon \re \eta_\xi \,\; \textnormal{(in the second integral).} $$

  Also observe that $\, (\overline{1 +\varepsilon \eta_\xi} )\, (1\,- 2\varepsilon \re \eta_\xi) \,=\, 1 - \varepsilon \,\eta_\xi \;+\;\textnormal{higher powers of}\,\varepsilon\,  $. Substituting these equations into (\ref{powerExpansion}), in view of $\,\int_\Omega h_\xi \overline{h_{\bar \xi}} \;\; \eta_{\bar{\xi}}\, \textbf{d} \xi \;=\;0\,$,  we conclude with  formula  (\ref{ExpansionInEpsilon}), as desired by Theorem \ref{Expansion}.

\section{The Case $\,\mathcal H  = h_\xi \overline{h_{\bar{\xi}}}\,\equiv\,0\,$}\label{SingularCase}

On the key issue of  Dirichlet Principle the following equation
 \begin{equation}\label{BorderlineEquation1}\,h_\xi\, \overline{h_{\bar\xi}}\, \equiv 0\,\;\;\;\;\textnormal{(homogeneous Hopf product)}
  \end{equation}
  lies around the borderline of behavior with respect to the inner variations, see Remark \ref{RemarknHarmonics}
  on $\,n\,$-dimensional variant of (\ref{BorderlineEquation1}). We call such solutions the \textit{singular Hopt harmonics}.
\subsection{Proof of Corollary \ref{theSingularCase}}\begin{proof}
 In the singular case formula  (\ref{ExpansionFormula}) simplifies as,
\begin{equation}\label{ReducedIdentity}
\mathcal E [H] - \mathcal E[h] \;=\; \int_{\Omega}  \frac{ 2 \, ( |h_\xi|^2 \,+\, | h_{\bar \xi} |^2 ) \;\abs{z_{\bar \xi}    }^2 \, \;\textbf d \xi   }{    \abs{z_\xi }^2 -  \abs{z_{\bar \xi}    }^2       }  \; \geqslant 0\,.
\end{equation}
 Hence Corollary \ref{theSingularCase} is immediate.
 \end{proof}

  The identity \ref{ReducedIdentity} also tells us that we have equality $\,\mathcal E(H) = \mathcal E(h)\,$  iff $\, z_{\bar \xi}(\xi) \cdot Dh(\xi) \equiv 0\,$. In terms of the inverse map $\,\xi = \xi(z)\,$ this condition reads as $\, \xi_{\bar z}(z) \cdot Dh(z) \equiv 0\,$.
Suppose for the moment  that
 $\,Dh(z) \not = 0\,$ almost everywhere.  Then all inner variations ($\xi \not \equiv 0\,$) strictly increase the energy. Indeed, otherwise we would have $\,\xi_{\bar z} (z) \equiv 0\,$ so the function $\,\xi = \xi(z)\,$ would be holomorphic and, being equal to $\,z\,$ near the boundary of $\,\Omega\,$, would be identically equal to $\,z\,$, resulting in no change of variables.\\
 Examples abound in which the Hopf product vanishes.

  \subsection{Origami Folding}
  Surprisingly, in  \cite{IVV} there has been constructed a Lipschitz map  $\,h : \mathbb C \into \mathbb C\,$ which vanishes in the lower-half plane $\,\mathbb C_- \bydef\{ \xi : \im \xi \leqslant 0\,\}\,$ and is a piecewise linear isometry on the upper-half plane $\,\mathbb C_+ \bydef\{ \xi : \im \xi \geqslant 0\,\}\,$. Precisely,  $\,\mathbb C_+\,$ has been triangulated so that on each of its triangles the differential $\,\textnormal d h   =  h_\xi \,\textnormal d \xi \,+ \,h_{\bar \xi} \, \textnormal d {\bar \xi}\,$  assumes one of the following six constant values.

  \begin{equation}\label{boundcond}
\begin{split}
\textnormal d h = \,\begin{cases}\textnormal d \xi\,,\,\;\,\,i \,\textnormal d \xi\,,\,\,\; -i\,\textnormal d \xi \,\qquad \mbox{(in which case $\,h\,$ is orientation preserving)}  &  \\
- \textnormal d \bar\xi\,,\, - i\,\textnormal d \bar\xi\,\,,\,\, i\,\textnormal d \bar\xi\,\, \;\;\mbox{( where $\,h\,$ is orientation reversing, foldings)}
\end{cases}\end{split}
\end{equation}
An interested reader is referred to an explicit construction by Formula (1.7) in Proposition 1 of \cite{IVV}.  Thus, at almost every $\,\xi\in \mathbb C\,$, we have either $\,h_\xi = 0\,$ or $\,h_{\bar\xi}  = 0\,$. Therefore, the Hopf product  vanishes  almost everywhere in $\,\mathbb C\,$. The change of orientation  of $\,h\,$ in $\,\mathbb C_+\,$ occurs more and more frequently when one approaches the common boundary $\,\partial \mathbb C_+ = \partial \mathbb C_- = \mathbb R\,$, \\
We have the following formulas:
 \begin{equation}
 \begin{split}
|\,h_\xi(\xi) \,|^2 \;+\, |h_{\bar \xi} (\xi)\,|^2\,\equiv  \; |\,J_h(\xi)\,| \,=\, \, \begin{cases}  1\;\;\textnormal{in}\;\;\mathbb C_+  &  \\
0\;\;\textnormal{in}\;\;\mathbb C_-
\end{cases} \end{split}
\end{equation}

 $$
\mathscr E[h] \;   \,=\,  \int_{\Omega_+} \left (|\,h_\xi(\xi) \,|^2 \;+\, |h_{\bar \xi} (\xi)\,|^2\,\right) \, \textbf{d} \xi\;\;=\;\; |\,\Omega_+ |\,,\;\;\textnormal{where}\;\; \Omega_+  \bydef \, \Omega \cap \mathbb C_+
$$

Consider an arbitrary  change of variables    $\,\xi = \xi(z)\,$  in $\,\Omega\,$  that equals $\,z\,$ near $\, \partial \Omega\,$.     Formula  (\ref{ExpansionFormula}), with $\,H(z) \bydef h(\xi(z))\,$,  cuts down considerably to:

\begin{equation}
\mathscr E [H] - \mathscr E[h] \;=\; \int_{\Omega_+}  \frac{ 2 \, \abs{z_{\bar \xi}    }^2 \, \;\textbf d \xi   }{    \abs{z_\xi }^2 -  \abs{z_{\bar \xi}    }^2} \geqslant  0.
\end{equation}
Equality occurs if and only if $\, z_{\bar \xi} \equiv 0\,$ on $\,\Omega_+\,$, meaning that $\,z = z(\xi) \,$ is holomorphic on $\,\Omega_+\,$. It then follows (by unique continuation property) that $\,z(\xi) \equiv \xi\,$ on $\,\Omega_+\,$.
\begin{remark} A natural question to ask is whether it is possible that, in spite of a  change of variables in $\,\Omega\,$,  the equation  $\,\mathscr E [H] \,=\,  \mathscr E[h]\,$ forces  $\,H\equiv  h \,$ on $\,\Omega\,$  .  The answer is "yes". To see such a possibility look at $\,H(z(\xi))  \equiv h(\xi)\,$ where the change of variables is given by the rule $\,z(\xi) = \xi\,$ for $\,\xi \in \mathbb C_+\,$,  so $\,H(\xi) = h(\xi)\,$, and  $\,z(\xi) = \xi\, +\,\eta(\xi)\,$ for $\,\xi \in \mathbb C_-\,$ with $\,\eta \in \mathscr C^\infty_0(\Omega_-)\,$. In this latter case, regardless of the choice of $\, \eta\,$,  both functions $\,H(\xi)\,$ and $\,h(\xi) \,$ vanish on $\,\Omega_-\,$.
\end{remark}
There is quite a general way to construct singular Hopf harmonics; typically, these are  piecewise holomorphic/antiholomorphic functions. The orientation of $\,h\,$ changes when passing through the   adjacent pieces of $\,\Omega\,$.

\subsection{Reflections about Circles}  Consider  a multiply connected  domain $\,\mathbb U\,$ with $\,(n-1)\,$ discs as bounded components of its complement, and the unit circle as its outer boundary, see Figure  \ref{CircReflections} on the left.  Reflect $\,\mathbb U\,$ about its inner boundary circles. This gives us $\,(n-1) \,$ circular domains $\,\mathbb U_1\,, ... ,\, \mathbb U_{n-1}\,$, each of connectivity $\,n\,$. Their outer boundaries are just the inner boundary circles of $\,\mathbb U\,$.
Next we reflect each $\,\mathbb U_1\,, ... ,\, \mathbb U_{n-1}\,$
about its own inner boundary circles. This gives us  $\,(n-1)^2\,$ circular domains of connectivity $\,n\,$, say $\,\mathbb U_{ij}\,$ with $ \, i, j = 1,2 ,..., n-1\,$, see Figure \ref{CircReflections} on the right.
 Continuing this process indefinitely, we cover the entire unit disc $\,\mathbb D\,$, except for a Cantor type limit set  $\,\mathcal C\,$, most desirably of zero measure.  Precisely, we have
$$\,\mathbb D \setminus\mathcal C = \bigcup_{i=1}^{n-1} \mathbb U_i \,\cup\, \bigcup_{i,j = 1}^{n-1} \mathbb U_{ij} \,\cup\, \bigcup_{i,j,k = 1}^{n-1} \mathbb U_{ijk}\,\cup ... \, $$
Our construction of the vanishing Hopf product $\,h_z \overline{h_{\bar z}} \equiv 0\,$ begins with an antiholomorphic function $\,\varphi(z) = \bar z\,$ in $\,\mathbb U\,$. We extend $\, \bar z\,$  to $\,\mathbb U_1\,\cup \, ... \,\cup \, \mathbb U_{n-1}\,$ by orientation reversing inversions  about the inner boundary circles of $\,\mathbb U\,$, respectively. This gives us orientation preserving linear fractional functions $\,\varphi_1 :\mathbb U_1 \onto \mathbb U\,, ... ,\, \varphi_{n-1} :\mathbb U_{n-1} \onto \mathbb U\,$. They  admit further circular inversions. Accordingly, we extend each $\,\varphi_i, \; i = 1,2, ... , n-1\,$ to $\,\mathbb U_{i1}\,\cup\,\mathbb U_{i2}\,\cup \, ... \,\cup  \mathbb U_{i, n-1}\,$ via the inversions about inner boundaries of $\,\mathbb U_i .\,$

Next, for every $\,i\,$,  we perform inversions about each inner boundary circle of $\,\mathbb U_i\,$. This gives us orientation reversing linear fractional transformations $\, \varphi _{i1} : \mathbb U_{i1} \onto \mathbb U_i\,,\; ...\; \varphi_{i,n-1} \onto \mathbb U_i\,$. Continuing this process indefinitely, we arrive at a map $\, h : \mathbb D \setminus \mathcal C  \onto \mathbb D \setminus \mathcal C \,$ with vanishing Hopf product, see also Figure \ref{MoreReflections} for more reflections.  The change of orientation  of $\,h\,$  occurs more and more frequently once we approach the limit set $\,\mathcal C\,$.
However, in general the energy of $\,h\,$ need not be finite.
\begin{figure}[!h]
\begin{center}
\includegraphics*[height=2.5in]{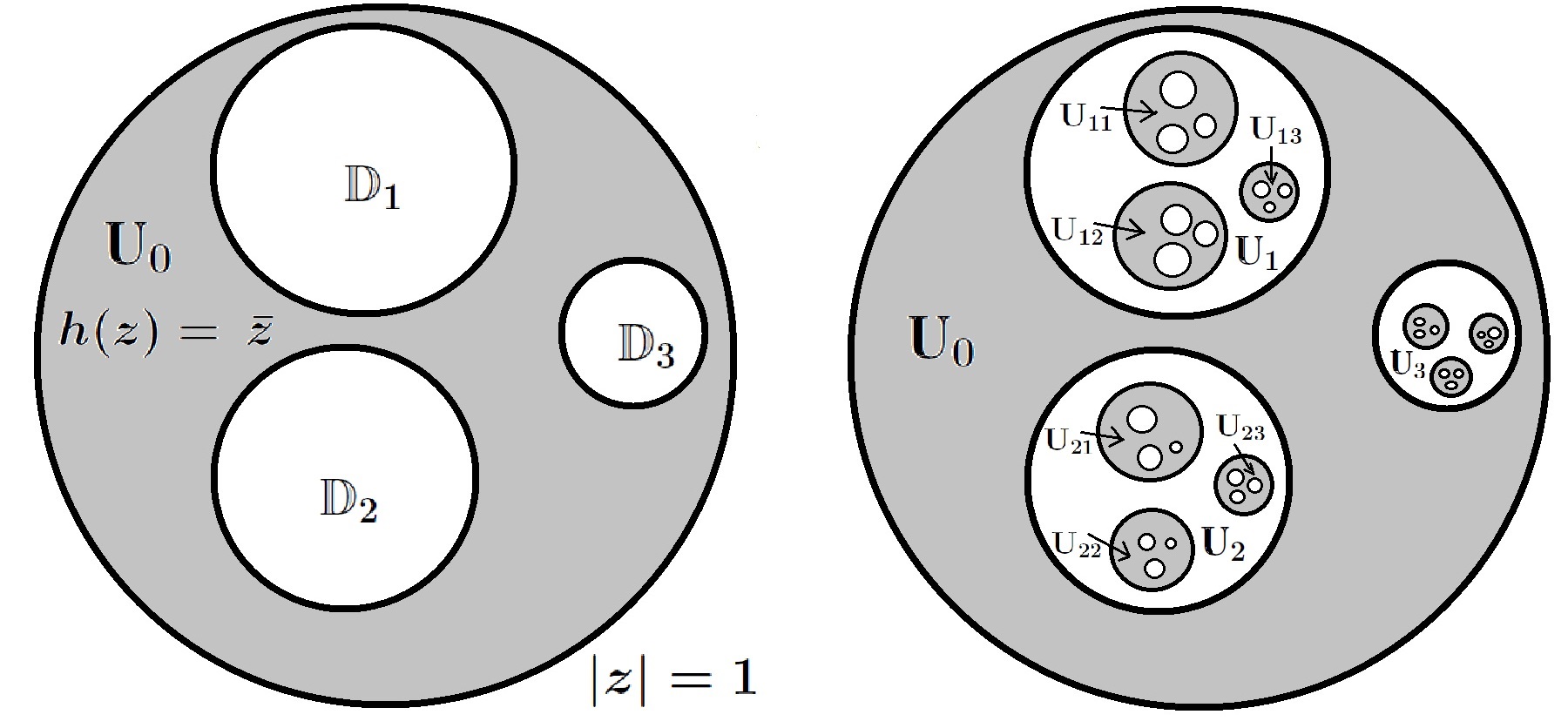}
\caption{Circular reflections result in a vanishing Hopf product.} \label{CircReflections}
\end{center}
\end{figure}

\begin{figure}[!h]
\begin{center}
\includegraphics*[height=2.5in]{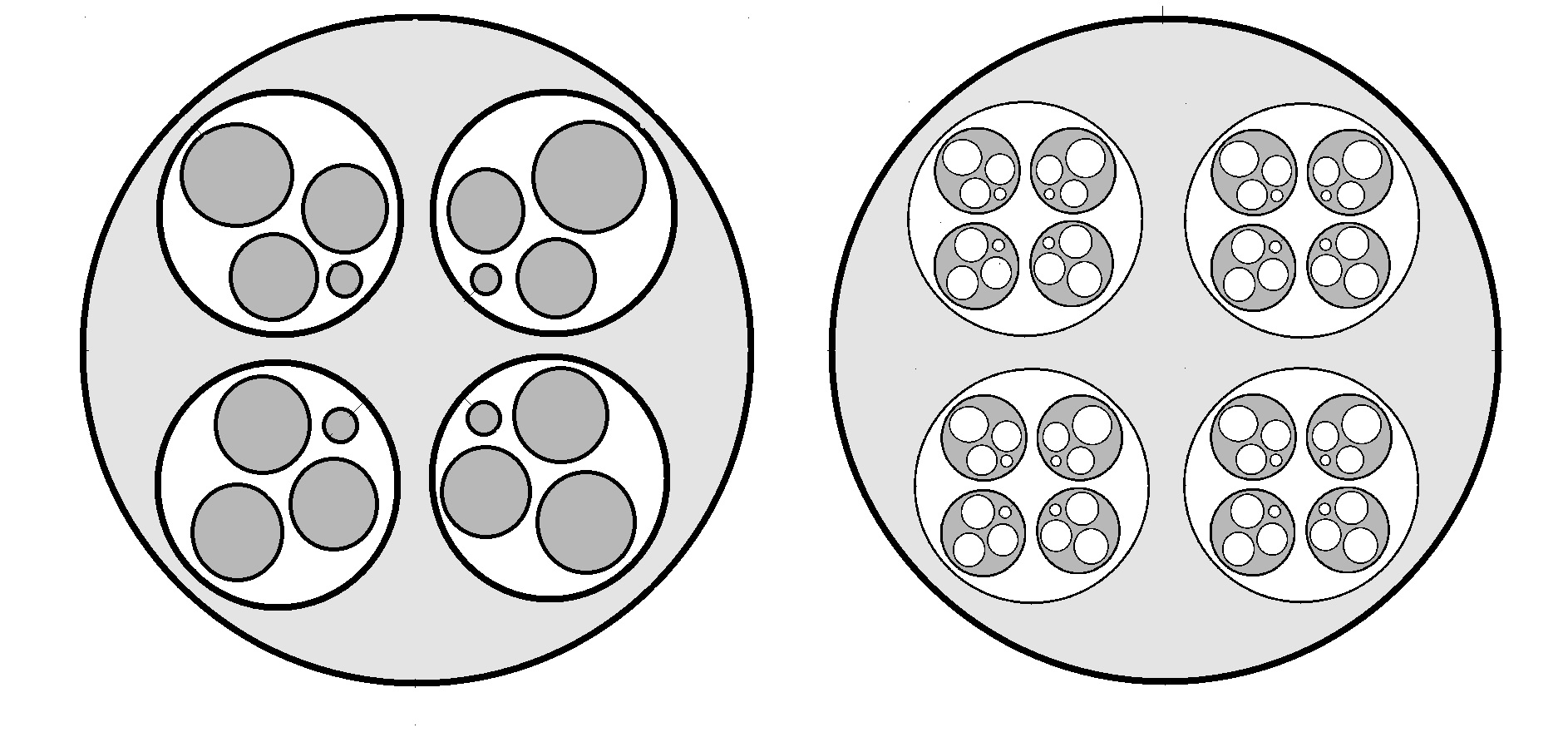}
\caption{Circular reflections of higher connectivity.}\label{MoreReflections}
\end{center}
\end{figure}

The simplest such a construction of finite energy  can be furnished via reflections about concentric circles, see Figure \ref{CircularReflections}.
For this purpose, we decompose the punctured disk into annuli $\, \mathbb D \setminus\{0\}\, =\, \mathbb A_1 \,\cup\, \mathbb A_2\,\cup\, ...\, \cup \mathbb A_n\,\cup\; \cdots\,$, where   $\,\mathbb A_n = \{ z \in \mathbb C : r_{n+1} \leqslant |z| < r_n\,\}\,$, with $\, r_n = n^{-2}\, $, for $\, n = 1, 2, ...\,$. We define the map $\,h = h(z)\,$ in the annulus $\,\mathbb A_n\,$ by the rule
 \begin{equation}
\begin{split}
h(z) = \,\begin{cases} n\bar z \,,\;\; \textnormal {for}\; \rho_n \leqslant |z| \leqslant r_n \;, \textnormal{where} \; \rho_n = \sqrt{\frac{1}{n(n+1)^3 }}\;\;\;\; ( \textnormal{thus}\;\; h_z \equiv 0 )  &  \\
\frac{1}{(n+1)^3 \;z} \;,\;\;\textnormal {for}\;\; r_{n+1} \leqslant |z| \leqslant \rho_n\, \;\;\;(\textnormal{thus}\; h_{\bar z} \equiv 0\;)
\end{cases}\end{split}
\end{equation}
The energy of $\,h\,$ in the annulus $\,\mathbb A_n\,$ is estimated as follows
\begin{equation}
\begin{split}
&\quad\quad\quad\quad\quad\mathscr E_{\mathbb A_n} [h] = \\ & \int_{r_{n+1} \leqslant |z| \leqslant \rho_n} |h_{z}|^2 \, \textnormal{\textbf{d}}z\;+\int_{\rho_n \leqslant |z| \leqslant r_n} |h_{\bar z}|^2 \, \textnormal{\textbf{d}}z \;\; \; \leqslant \frac{\pi}{(n+1)^2}  + \frac{\pi}{n^2} \,, \;\textnormal{respectively.}
\end{split}
\end{equation}
Summing up these estimates , we obtain $\,\mathscr E_\mathbb D [h]   < 2 \pi \sum_1^\infty \frac{1}{n^2} = \frac{\pi^3}{3}.\,$

\begin{figure}[!h]
\begin{center}
\includegraphics*[height=2.5in]{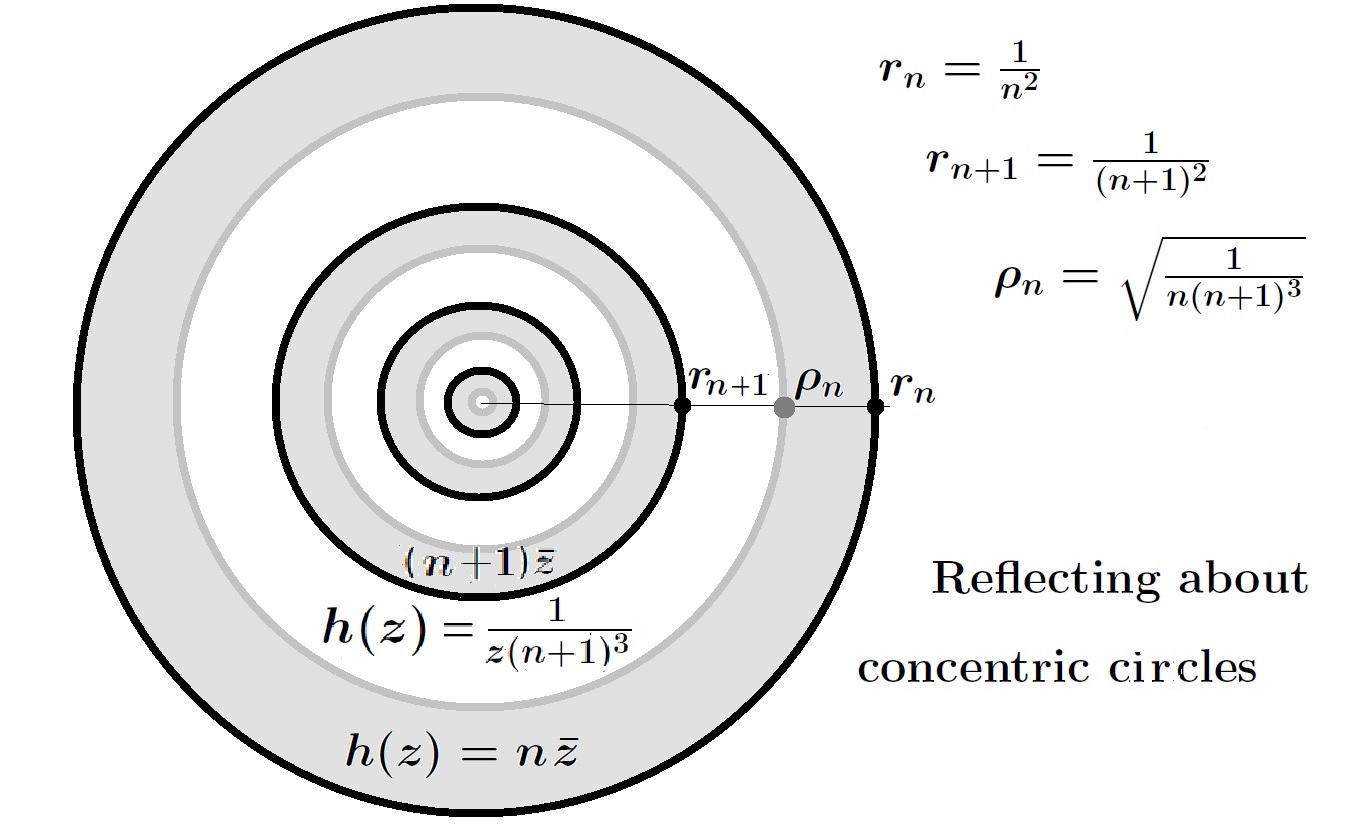}
\caption{Construction of a vanishing Hopf product for a map of finite energy.}\label{CircularReflections}
\end{center}
\end{figure}
\begin{remark} \label{RemarknHarmonics}
The theory of $\,n\,$-dimensional quasiconformal mappings is concerned with mappings $\,h : \mathbb X   \onto \mathbb Y\,$ of finite \textit{$\,n\,$-harmonic energy}, also called conformal energy.
\begin{equation}
\mathscr E[h] \;\bydef \int_\mathbb X  |Dh(x)|^n \,\textnormal d x     \; < \infty.
\end{equation}
Let us see how the associated variations might look like by analogy with the complex case.
The outer variation results in the \textit{$\,n\,$-harmonic equation}
\begin{equation}
\textnormal {div}\,  | Dh |^{n-2} Dh  \equiv 0
\end{equation}
The inner variations bring us to what we call \textit{Hopf $\,n\,$-harmonics} ~\cite{IOan}. These are  $\,\mathscr W^{1,n}_{\textnormal{loc}}\,$- solutions of the equation
\begin{equation}
\textnormal {div}\,  \left( | Dh |^{n-2} \Big[ \;D^*h\cdot Dh \; -\;\frac{1}{n}\, |Dh|^2\,\textbf{I} \;\Big]  \right)  \;=    0
\end{equation}
The radial squeezing $\,h(x) = \frac{x}{|x|}\,$ turns out to be Hopf $\,n\,$-harmonic in an annulus, but not $\,n\,$-harmonic.
An exact analogue of the singular Hopf harmonic equation (\ref{BorderlineEquation1}) \, takes the form,
\begin{equation}
 D^*h\cdot Dh \; -\;\frac{1}{n}\, |Dh|^2\,\textbf{I} = 0.
\end{equation}
Among other solutions are the conformal inversions about the $\,(n-1)\,$-spheres, both orientation reversing and orientation preserving.  All our   planar constructions presented above can be carried over to singular Hopf $\,n\,$-harmonics as well.
\end{remark}

\subsection{Solutions that are nowhere holomorphic and nowhere antiholomorphic}
We shall construct a Lipschitz solution to the equation $\,h_z \overline{h_{\bar z}}  =  0\,$ in $\,\mathbb C \,$ which is neither holomorphic nor antiholomorphic   in any open subset of $\,\mathbb C\,$. When constructing $\, h : \mathbb C \rightarrow \mathbb C\,$ we shall be dealing with two measurable sets \\

$
\mathbb F_+ \bydef \{ z ;\; \frac{\partial h}{\partial z} =  0 \} \;,\;\textnormal{and}\;\; \mathbb F_- \bydef \{ z ;\; \frac{\partial h}{\partial \bar z} =  0 \} \;\;\;
$\\

Here $\,z\,$ runs over the points of differentiability of $\,h\,$; thus a set of full measure in $\,\mathbb C\,$.
\begin{proposition} \label{StrangeHarmonic}The construction reveals the following properties.
\begin{itemize}
\item  The sets  $\,\mathbb F_+\,$ and $\,\mathbb F_- \,$ are "measure disjoint"; that is, $\,\mathbb F_+ \cap \mathbb F_-\,$ has zero measure.
    \item The union $\,\mathbb F_+ \cup \mathbb F_-\,$ has full measure on $\,\mathbb C\,$.
    \item Both $\,\mathbb F_+\,$ and $\,\mathbb F_-\,$ are "measure dense" in $\,\mathbb C\,$; meaning that,\\

    $  |\mathbb F_+ \cap \Omega |  > 0 \,\;\textnormal{and}\; |\mathbb F_- \cap \Omega |  > 0 \,, \;\textnormal{for every open set}\;\Omega \subset \mathbb C. $
\end{itemize}
\end{proposition}
\begin{proof}
The map in question will be  defined by the rule $\, h(z) =  u(x)  +  i  y  \;$, for $\,z =  x+ i y\,$.  Here $\,u = u(x)\,$ is a Lipschitz function on $\,\mathbb R\,$ whose derivatives at the points of differentiability assume only two values $\,\pm 1\,$, say $\,+1\,$ on a set $\,E_+\,$ and $\,-1\,$ on a set $\,E_-\,$. Moreover,
\begin{itemize}
\item $\,|\,E_+ \cap  E_- | = 0\,$,  and the union $\,E_+ \cup E_- \,$ has full measure in $\,\mathbb R. \,$
\item  The sets $\,E_+\,$ and $\,E_-\,$ are "measure" dense on $\,\mathbb R\,$; meaning that $\,|E_+\cap I | > 0\,$ and $\,|E_-\cap I | > 0\,$, for every open interval $\,I\subset \mathbb R\,$.
\end{itemize}
A construction of such sets $\,E_+\,$ and $\,E_-\,$, known as well-distributed measurable sets, can be found in  \cite{Simoson, Rudin}. Having those sets in hand we define:
\begin{equation}\
\begin{split}
 u(x) \bydef \int_0^x \,\chi(t)\,\textnormal{d} t\;,\;\;\textnormal{where}\;  \chi(t) =
\,\begin{cases} +1 \;, \; \textnormal {on}\; E_+\;  &  \\
-1 \;, \;\textnormal {on}\;\;E_-
\end{cases}\end{split}
\end{equation}
Obviously $\,u\,$ is Lipschitz continuous, so differentiable almost everywhere. Cut slightly those sets to obtain:
\begin{itemize}
\item  $\,{{E}^{'}_+}\,$ - the set of density points of $\,E_+\,$ at which $\,u\,$ is differentiable.
\item  $\,E_-^{'}\,$ - the set of density points of $\,E_-\,$ at which $\,u\,$ is differentiable.
\end{itemize}
We readily infer from these definitions that
\begin{equation}\nonumber
\begin{split}
 u^{,}(x) \,=\, \lim_{\varepsilon \rightarrow 0}\;\frac{u(x +\varepsilon) - u(x)}{\varepsilon}\;= \lim_{\varepsilon \rightarrow 0} \frac{1}{\varepsilon} \int_x^{x+\varepsilon} \chi(t)\,\textnormal d t =
\,\begin{cases} +1 \;, \; \textnormal {on}\; E^{'}_+\;  &  \\
-1 \;, \;\textnormal {on}\;\;E^{'}_-
\end{cases}\end{split}
\end{equation}
as desired.\\
Now the sets in Proposition \ref{StrangeHarmonic} are given by $$\,\mathbb F_+ \bydef\, E_+ ' \times \mathbb R\,\;\;\textnormal{and}\;\;\,\mathbb F_-\bydef\,  E_- ' \times \mathbb R\ .\,$$
The computation of complex derivatives of $\,h\,$ runs as follows:
$$ h_z \;=\; \frac{1}{2} ( h_x - i h_y)  = \frac{1}{2} [ u'(x) + 1 ]  =  0 \;\;\textnormal{on}\;\; \overline{\mathbb Z} \,. $$
$$ h_{\bar z} \;=\; \frac{1}{2} ( h_x +\,i h_y)  = \frac{1}{2} [ u'(x) - 1 ]  =  0 \;\;\textnormal{on}\;\; \overline{\mathbb Z} \,.$$

Furthermore,
$$  |Dh(z) |^2  \bydef |h_z|^2  + |h_{\bar z}|^2  =  \frac{1}{2} \left( [u'(x)]^2 \;+ \,1 \right) \,\equiv  1\;,\;\textnormal{almost everywhere}\,. $$

$$  J_h(z)  =  |h_z|^2  - |h_{\bar z}|^2  =  u'(x) \pm 1 \;,\;\textnormal{almost everywhere}\,. $$
\end{proof}
\begin{remark} Analogously, in higher dimensions, one may consider the singular Hopf $\,n\,$-harmonic map
\begin{equation}
h(x_1, x_2,\, ... ,\, x_n ) \; = (u(x_1) , x_2,\, ... \,, x_n)\;,\;\; D^*h\cdot Dh \; = \frac{1}{n} | Dh|^2 \,\textbf{I}.
\end{equation}
\end{remark}
\begin{remark}
  Complete description of $\,\mathscr W^{1,2}_{\textnormal{loc}} (\Omega)\,$ -solutions to the singular Hopf equation $\,h_z \overline{h_{\bar z}}  = 0 \,$ remains open.
  \end{remark}

\section{Proof of Lemma \ref{InequalityForH} in case $\,\mathcal H(\xi)  = \mathcal A^2(\xi)\,$} \label{Section5}
For the sake of clarity, before we present the full proof of Lemma \ref{InequalityForH}, let us first demonstrate the case when $\,\mathcal H\,$ admits a continuous  branch of the square root, say $\, \sqrt{\mathcal H} =   \mathcal A\,$.

Consider a mapping $\,f \bydef \eta \, \mathcal A \,: \,\Omega \rightarrow \mathbb C\,$. As a starting point,  we record the identity
\begin{equation}\label{AnIdentity1}
\begin{split}\int _\Omega \mathcal A \mathcal A'\, \eta\, \eta_{\bar \xi} \,\, \textbf{d} \xi &= \frac{1}{4} \int _{\Omega} [\mathcal H'\, \eta^2\, ]_{\bar \xi} \, \, \textbf{d} \xi = \frac{i}{8} \,\int _\Omega [\mathcal H'\, \eta^2\, ]_{\bar \xi} \,\,\textnormal d \xi \wedge \textnormal d \bar{\xi} \\& =   \frac{-i}{8} \int _\Omega \textnormal d \left (\mathcal H'\, \eta^2 \,\textnormal d \xi \right )  =   \frac{-i}{8} \int _{\partial \Omega}  \mathcal H'\, \eta^2 \, \textnormal d \xi \; = \, 0
\end{split}
\end{equation}
because $\, \mathcal H'\, \eta^2 = 0 \,\,\textnormal {on}\,\,  \partial \Omega\,$ .

Now the computation runs as follows:
\begin{equation}\label{Chain}
\begin{split}
 \int_\Omega   |\mathcal H(\xi)|  \; \abs{\eta_{\bar \xi}}^2
& = \,\int _{\Omega} | \mathcal A |^2 | \eta_{\bar \xi}|^2  = \,\int _{\Omega} | f_{\bar \xi}|^2 \\
& \; =  \frac{1}{2}  \,\int _{\Omega} \left (| f_{\bar \xi}|^2 \,+\, | f_{\xi}|^2\right)\; \;\;\;\big(\textnormal{due to} \,\int_\Omega J_f(\xi) \, \textbf{d} \xi = 0)\;\big) \\ & \geqslant \,\int _{\Omega} | f_{\bar \xi} f_{\xi}|  \geqslant \,\left |\,\int _{\Omega}  f_{\bar \xi} f_{\xi}\,\right | \\ &  = \left | \,\int _{\Omega}  \left(\mathcal A \,\eta_{\bar \xi} \right)\; \left(\mathcal A \eta_\xi \,+\, \mathcal A' \eta\right)\,\right | = \left |\,\int _{\Omega}  \left(\mathcal A \,\eta_{\bar \xi} \right)\; \left(\mathcal A\eta_\xi \,+\, \mathcal A' \eta\right)\,\right | \\ & = \left | \, \int _{\Omega}  \left(\mathcal A\,\eta_{\bar \xi} \right)\; \left(\mathcal A \eta_\xi \right)\,\right |\;\;\;\big(\textnormal{due to identity (\ref{AnIdentity1})}\;\big) \\ &  =\,  \left |\,\int _{\Omega}  \mathcal A^2 \,\eta_{\bar \xi} \; \eta_\xi\,\right |    = \,  \left |\,\int_\Omega \mathcal H(\xi)  \;\eta_\xi\, \eta_{\bar \xi}\,\right |
\end{split}
\end{equation}
as desired.

\section{A partition into Rectangles} \label{RectangularPartition}
For the full proof of  Lemma  \ref{InequalityForH},  we need additional geometric considerations to deal with the lack of continuous square root of the Hopf product $\,\mathcal H(\xi)  = h_\xi \overline{h_{\bar{\xi}}}\,$.

 Suppose we are given a domain $\,\Omega \subset \mathbb C\,$, a compact subset $\, \textit{\textbf{K }}\subset \Omega\,$ and a finite set $\,\mathcal Z = \{ z_1, z_2, ... , z_n\,\} \subset  \textit{\textbf{K}}\,$. In the applications $\,\mathcal Z\,$ will consist of zeros of a holomorphic Hopf product $\,\mathcal H(z) = h_z \overline{h_{\bar{z}}}\,$ . The goal is to construct disjoint simply connected domains $ R_1, \, R_2, \, ... \,,\, R_N\,$,  whose closures are contained in $\,\Omega\,$ and cover  $\,\textit{\textbf{K}}\,$.   In symbols,
$$ \;\;\textit{\textbf{K}} \subset  \overline{R_1} \,\cup \,\overline{R_2}\, \cup \,...\,\cup\overline{R_N} \,\subset \Omega \,\;\textnormal{and}\;\; R_\alpha\,\cap \, R_\beta\, = \emptyset\,\,\textnormal{for}\,\;  \alpha , \beta =  1,2, ... , N\,,\; \alpha \not =\beta\;$$
It will also be required that for every pair  $\,\{ R_\alpha , R_\beta \}_{\alpha \not =\beta} \,$ the intersection  $\,\overline{R_\alpha}\,\cap \overline{\, R_\beta}\,$ is either empty, a single point called \textit{corner} of the partition, or a $\,\mathscr C^1\,$-regular closed Jordan arc denoted by $\,\Gamma_{\alpha \beta} \bydef \overline{R_\alpha} \cap  \overline{R_\beta }\,$. This is the common side of $\, R_\alpha \,$ and $\, R_\beta \,$. We    refer to such $\, R_\alpha \,$ and $\, R_\beta \,$ as side-wise\,\textit{adjacent} domains. Furthermore, each point  $\,  z_1, z_2, ... , z_n\,$ is a corner of the partition and, as such, does not lie in any of the domains  $ R_1, \, R_2, \, ... \,,\, R_N\,$.
\begin{remark}
 For our purposes here,  the simplest way to build such a partition is to take for $ \,R_1,\, ... \,, R_N\,$ coordinate rectangles (sides parallel to the $\,x, y\,$  coordinate axes). However, for various specific purposes, the theory of critical horizontal and vertical trajectories of the Hopf quadratic differential $\, \mathcal H(z)\, \textnormal dz \otimes \textnormal dz\,$  (as sides of the domains $\,R_\alpha\,$)  gives us a tool of much wider applicability, see Section \ref{LengthArea}  and Figure \ref{Figure2}.
\end{remark}

\subsection{A Rectangular Partition}

 Choose and fix an $\,\varepsilon > 0\,$ small enough so that  $\,\textnormal{dist} (\textit{\textbf{K}} , \partial \Omega) >  2 \varepsilon\,$. As a first step, we divide $\,\mathbb R^2\,$ into squares of side-length $\,\varepsilon\,$ by cutting $\,\mathbb R^2\,$  along the horizontal lines $\, \{ (x, i\varepsilon) \in \mathbb R^2\, : x \in \mathbb R\,\}\,$, $\,i = 0, \pm 1, \, \pm 2,\,...\,$, and the vertical lines  $\, \{ (j\varepsilon, y) \in \mathbb R^2\, : y \in \mathbb R\,\}\,$, $\,j = 0, \pm 1, \, \pm 2,\,...\,$. This gives us an $\,\varepsilon\,$- mesh of Cartesian squares,
$$
\mathscr M_\varepsilon \,\bydef \, \{Q_{ij} \,\}_{i,j \in \mathbb Z} \, , \; \textnormal{where} \; Q_{ij} \,=\, \{ (x,y) \colon  \; i\varepsilon < x < i\varepsilon + \varepsilon\;\,,\, j \varepsilon < y < j\varepsilon  + \varepsilon \; \}
$$
It is not generally possible to  construct a mesh of Cartesian squares whose corners cover all points $\, z_1, z_2, ..., z_n\,$; we need additional (finite number)  horizontal and vertical cuts of $\,\mathbb R^2\,$.  Through every point $\, z_\nu = x_\nu +  i y_\nu\,\in \mathcal Z\,,\, \nu = 1, 2, ... , n\,$, there pass two lines: a horizontal line  $\,\{ (x, \, y_\nu) :\, x \in \mathbb R\,\}\,$,  and the vertical line $\, \{ (x_\nu \,, \,y ) : \,y \in \mathbb R\,\}\,$. Removing all these lines (additional cuts together with the ones for the $\,\varepsilon\,$- mesh of Cartesian squares),  leaves us a family of open rectangles.  Let us denote this family by $\,\mathscr M_\varepsilon(z_1, z_2, ... , z_n)\,$. Clearly, $\,\mathscr M_\varepsilon(z_1, z_2, ... , z_n)\,$ is a refinement of $\,\mathscr M_\varepsilon\,$. It then follows that each side of a rectangle in $\,\mathscr M_\varepsilon(z_1, z_2, ... , z_n)\,$ is shorter or equal to $\,\varepsilon\,$. Let us record this observation as:
$$
\;\textnormal{diam}\, R \leqslant \sqrt{2}\,\varepsilon  < 2\, \varepsilon\;,\; \textnormal{for every}\; R \in \mathscr M_\varepsilon(z_1, z_2, ... , z_n)\,.
$$
Therefore, whenever the closure $\,\overline{R}\,$ of a rectangle $\, R \in \mathscr M_\varepsilon(z_1, z_2, ... , z_n)\,$ intersects $\, \textit{\textbf{K}}\,$ it lies entirely in $\,\Omega\,$. Now comes the construction of the desired family $\,\mathscr F \bydef \{  R_1, \, R_2, \, ... \,, R_N\}\,$,

\begin{definition}{\label{FamilyF}} The family $\,\mathscr F \bydef \{  R_1, \, R_2, \, ... \,, R_N\}\,$ consists of all open rectangles in $\,\mathscr M_\varepsilon(z_1, z_2, ... , z_n)\,$ whose closers intersect $\, \textit{\textbf{K}}\,$\end{definition}

 \begin{figure}[!h]
\begin{center}
\includegraphics*[height=2.5in]{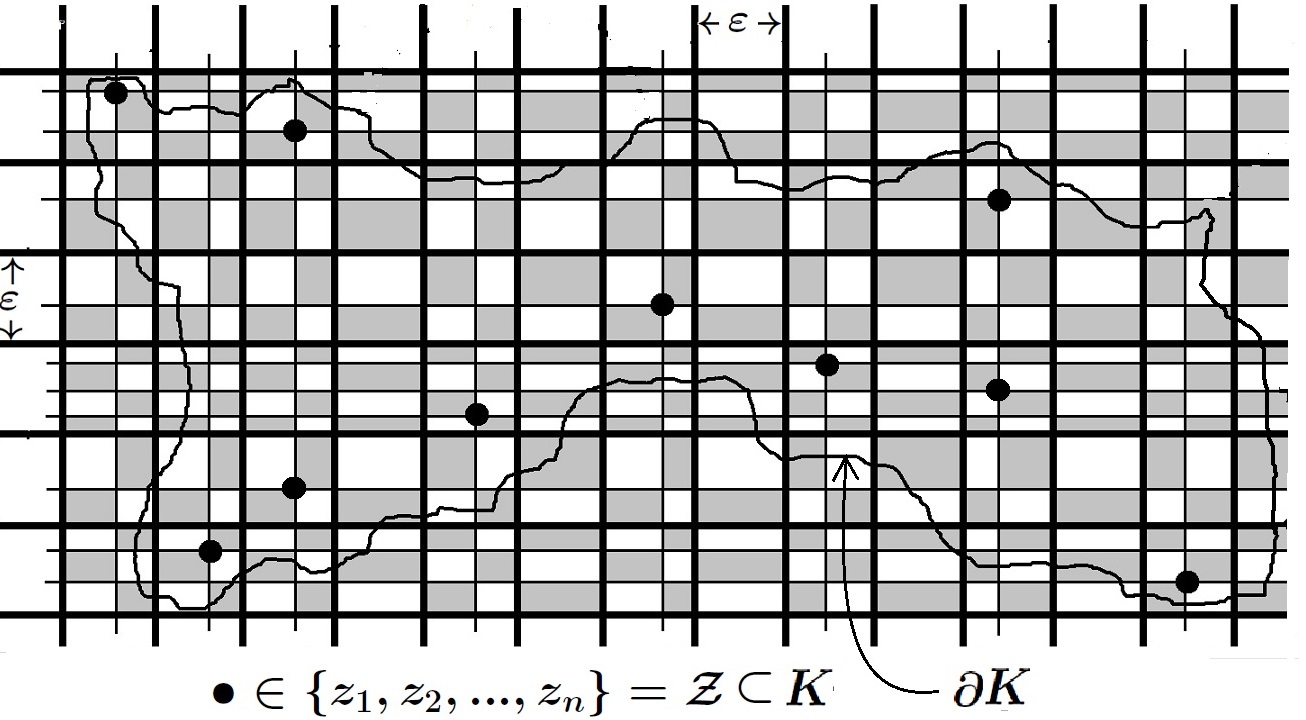}
\caption{The $\varepsilon\,$-mesh of squares in $\,\mathbb R^2\,$, its refinement $\,\mathscr M_\varepsilon(z_1, z_2, ... , z_n)\,$ and the family $\,\mathscr F\,$ of selected rectangles.}\label{Figu1}
\end{center}
\end{figure}
Let us take a look at the sides (horizontal and vertical) of  rectangles in $\, \mathscr F\,$. Every such a side, denoted in a generic way by $\,\Gamma\,$,   either lies entirely in $\, \Omega \setminus \textit{\textbf{K}}\,$ or is a common side of two adjacent rectangles, say $\, \Gamma =\overline{ R_\alpha} \cap \overline{R_\beta}\,$ for some $\, R_\alpha , R_\beta\,\in \,\mathscr F\,$. In this latter case there comes an issue of orientation.\\
Every rectangle $ \, R \in \mathscr M_\varepsilon( z_1, z_2, ... z_n )\,$ will be oriented positively with respect to the orientation of $\,\mathbb R^2\,$. This gives us the so-called positive (with respect to $\,R\,$) orientation of $\, \partial R\,$. Geometrically, traveling along $\, \partial R\,$ in the positive direction (counterclockwise) the rectangle $\, R\,$ remains on the left hand side.\\
Consider a pair of \textit{side-wise} adjacent rectangles $\, R_\alpha , R_\beta \in \mathscr M_\varepsilon( z_1, z_2, ... , z_n )\,$ and their common side $\,\Gamma =\Gamma_{\alpha \beta} \bydef \overline{R_\alpha} \cap  \overline{R_\beta }\,$.  When $\,\Gamma\,$ is positively oriented with respect to $\,R_\alpha\,$, we indicate it by writing  $\,\Gamma = \Gamma_\alpha^\beta\,$. Accordingly, $\,\Gamma_\beta ^\alpha\,$, being positively oriented with respect to $\,R_\beta\,$,  is negatively oriented  with respect to $\,R_\alpha\,$.  In other words, $\, \Gamma_\alpha^\beta\,$\,and $\,\Gamma_\beta ^\alpha\,$ have opposite orientation.
\begin{figure}[!h]
\begin{center}
\includegraphics*[height=2.5in]{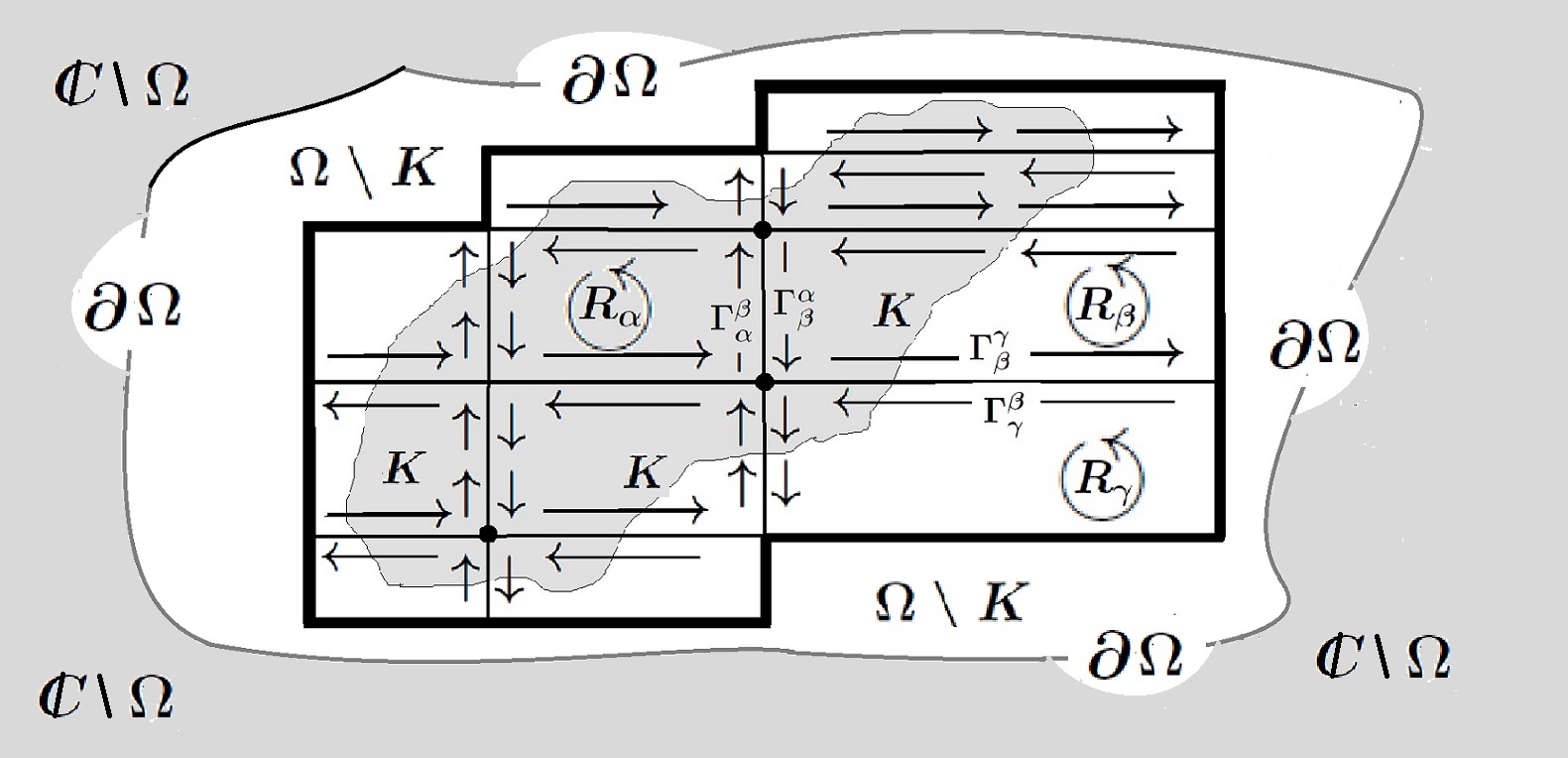}
\caption{Covering a compact subset $\,\textit{\textbf{K}} \Subset  \Omega\,$ by oriented rectangles  in $\,\Omega\,$. Their adjacent sides have opposite  orientation.}\label{Figure2}
\end{center}
\end{figure}

The above family $\,\mathscr F\,$ of oriented rectangles is particularly convenient when integrating an exact differential 2-form  $\, d\,\omega\,$,  where the 1-form $\omega\,$ can only be locally defined. This is typical when one needs to select locally defined branches of $\,\omega\,$  differing in sign. An analogy to taking square root of a holomorphic quadratic differential can be found, see also \cite{HubbardMasur}  for far reaching abstraction.  Let us look at a particular situation of this kind.

\subsection{2-valued mappings} \label{2valeudMappings}

 Recall that we are given a holomorphic function\\$\,\mathcal H \not \equiv 0\,$  in  $\, \Omega \subset \mathbb C \,$ and a complex valued function $\, \eta \in \mathscr C^\infty_0( \Omega)\,$ with compact support in $\, \textit{\textbf{K }}\Subset \Omega\,$. In particular,  $\,\mathcal H\,$ has only a finite number of zeros in  $\, \textit{\textbf{K }}\subset \Omega\,$, say $\,\mathcal Z = \{ z_1, z_2, ... , z_n\,\} \subset  \textit{\textbf{K}}\,$.
\\
 Suppose, as a starting point, that $\, \mathcal H \not = 0 \,$  in some simply connected Lipschitz subdomain $\,R \Subset  \Omega\,$. Thus, in particular, $\,\mathcal H\,$ admits a continuous branch of square root therein; precisely,  $\,H = \mathcal A^2\,$ for a function $\,\mathcal A \,$ that is continuous in $\,\overline{R}\,$ and holomorphic in $\,R\,$. We consider the mapping $\, f \bydef  \mathcal A \, \eta\, = \pm\sqrt{\mathcal H } \,\eta\, $.  Its Jacobian area-form  does not depend on the choice of the $\,\pm$sign for $\,\sqrt{\mathcal H}\,$ ; precisely,
 \begin{equation}
 2 i\, J_f(z)\,\mathbf d z \,= \, \textnormal d \overline{f}  \, \wedge \textnormal d f  \,= \, \textnormal d \, (\overline{f}  \, \wedge \textnormal d f)\; =\;\textnormal d \omega
 \end{equation}
 where $\,\omega\,$ is a  differential 1-form  given on $\,R\,$ by the rule
 \begin{equation}{\label{omegaform}}
 \omega \,\bydef \, \overline{f}  \, \wedge \textnormal d f\; =\;  \frac{ | \eta|^2\,\,\overline{\mathcal H} \,\textnormal d \mathcal H\;\,  +\;\, 2 |\mathcal H |^2\,\, \overline{\eta}\, \textnormal d  \eta }{2\,|\mathcal H |} \,\;
 \end{equation}
It should be noted that the latter expression defines a differential 1-form, still denoted by $\,\omega\,$, on the entire domain $\Omega \setminus \mathscr Z\,$ irrespective of which $\,\pm$ sign for $\,f\,$  is used.
 Also note that $\,\omega\,$ is bounded and $\,\mathscr C^\infty\,$-smooth in $\Omega \setminus \mathscr Z\,$. This makes it legitimate to apply integration by parts.
\begin{equation}
\iint_R\,  J_f(z)\,\mathbf d z = \frac{1}{2 i} \int_{\partial R} \,\omega \;\;\;\;(\partial R \,\,\textnormal { - oriented counterclockwise})
\end{equation}

 We shall now make use of the family  $\,\mathscr F = \{\, R_1, R_2, ...\,, R_N\,\}\,$  of rectangles as simply connected domains in which $\,\mathcal H \not= 0\,$.  On each $\,R \in \mathscr F\,$   we are at liberty to choose a continuous branch of  $\,\sqrt{\mathcal H}\,$. Once this is done, we obtain a family of mappings $\,f^\alpha : R_\alpha \rightarrow \mathbb C\,$, $\,\alpha = 1, 2, ... , N \,$, defined by the rule $\,f^\alpha = \sqrt{\mathcal H} \;\eta\,$ where the branch of $\,\sqrt{\mathcal H}\,$ depends on $\,\alpha\,$, whence the superscript $\,\alpha\,$.\\
 Before proceeding further in this direction, assume for the moment that  $\,\mathcal H\,$  admits continuous square root in the entire domain $\,\Omega\,$, so that  $\,f \in \mathscr C^\infty_0(\Omega)\,$.  Consequently,  $\,\iint _\Omega J_f(z) \,\mathbf d z  =  0\,$, where we recall that  $\,J_f(z) = |f_z|^2  \,-\, |f_{\,\overline{z}}|^2\,$. In our more general setting the above ideas still work to give similar identity.

 \begin{lemma}{\label{IntegralsOfJacobians}}
 Due to the cancellation of boundary integrals we have,
 \begin{equation} \label{SumIntegrals}
 \iint _{R_1} J_{f^1}(z) \,\mathbf d z  \;+\; \iint _{R_2} J_{f^2}(z) \,\mathbf d z \;+\; \cdots \;+\iint _{R_N} J_{f^N}(z) \,\mathbf d z  \; = \;0
 \end{equation}
 Equivalently,

 \begin{equation}\label{sumOfsquares}
 \iint _{R_1} |f^1_z|^2  \;+\; \cdots \;+\iint _{R_N} |f^N_z|^2   \; = \;\iint _{R_1} |f^1_{\,{\overline z}}|^2  \;+\; \cdots \;+\iint _{R_N} |f^N_{\, \overline{z}}|^2
 \end{equation}

 \end{lemma}

  \begin{proof}
 Upon integration by parts, each integral over $\,R_\alpha\,$ in (\ref{SumIntegrals})\,, $\,\alpha = 1, 2, ... , N\,$,  takes the form
 $$ \,\iint _{R_\alpha} J_{f^\alpha}(z) \,\mathbf d z  =  = \frac{1}{2 i} \int_{\partial R_\alpha} \,\omega \;\;\;\;(\partial R_\alpha \,\,\textnormal { is oriented counterclockwise}) \, $$
  where $\,\omega\,$  is independent of $\,\alpha\,$, see formula at (\ref{omegaform}).  We are reduced to showing that
  \begin{equation}
  \sum_{\alpha = 1}^N \int_{\partial R_\alpha} \,\omega \;= \,0
  \end{equation}

  The oriented boundary  of the rectangle  $\,R_\alpha\,$ consists of four oriented straight line segments. There  is nothing to integrate over a segment that lies entirely in $\,\Omega \setminus \mathbf K\,$, because  $\,\omega \equiv 0\,$ therein.
   Therefore, we need only consider the segments that intersect $\,\mathbf K\,$. These segments are exactly the common sides of two side-wise adjacent rectangles in the family  $\,\mathscr F = \{\, R_1, R_2, ...\,, R_N\,\}\,$, which  is immediate from our definition of $\,\mathscr F\,$. In other words, we are reduced to showing that

   \begin{equation}
  \sum_{\alpha \not = \beta}^N \;\int_{\Gamma_\alpha^\beta\,}  \,\omega \;= \,0
  \end{equation}
      Here $\,\Gamma_\alpha^\beta\,$ and $\,\Gamma_\beta^\alpha\,$ represent the same straight line segment $\, \overline{R_\alpha}\cap\,\overline{R_\beta}\,$,  but with opposite orientation. This results in  $\,\int_{\Gamma_\alpha^\beta\,}  \,\omega \; + \;\int_{\Gamma_\beta^\alpha\,}  \,\omega = 0 \,$, completing the proof of Lemma \ref{IntegralsOfJacobians}.

  \end{proof}

\section{Proof of Lemma \ref{InequalityForH} and Theorem \ref{SecOrdInequality}}\label{Section7}

\subsection{Proof of Lemma \ref{InequalityForH}}

Recall from Section \ref{RectangularPartition} the family  $\,\mathscr F \bydef \{  R_1, \, R_2, \, ... \,, R_N\}\,$  of rectangles. Let  $\,\mathcal A_\alpha\,$ be a continuous branch of $\,\sqrt{\mathcal H}\,$ in $\,R_\alpha\,$, $ \alpha = 1, 2, ... , N\,$; that is, $\,\mathcal A_\alpha ^2 = \mathcal H\,$ in $\, R_\alpha\,$.  Also  recall the mappings $\,f^\alpha : R_\alpha \rightarrow \mathbb C\,$  defined by the rule $\,f^\alpha = \mathcal A_\alpha \;\eta\,$.

First note the following identity
\begin{equation}\label{AnIdentity}
\sum_{\alpha = 1}^N\,\int _{R_\alpha} \mathcal A_\alpha \mathcal A_\alpha'\, \eta\, \eta_{\bar \xi}  = \frac{1}{4} \sum_{\alpha = 1}^N\,\int _{R_\alpha} [\mathcal H'\, \eta^2\, ]_{\bar \xi} =  \frac{1}{4} \,\int _\Omega [\mathcal H'\, \eta^2\, ]_{\bar \xi} = 0,
\end{equation}
because $\, \mathcal H'\, \eta^2 \in \mathscr C_0^\infty (\Omega)\,$ .
\newpage
Now the computation runs as follows:
\begin{equation}\label{Chain}
\begin{split}
 &\int_\Omega   |\mathcal H(\xi)|  \; \abs{\eta_{\bar \xi}}^2\\
& = \sum_{\alpha = 1}^N\,\int _{R_\alpha} | \mathcal A_\alpha |^2 | \eta_{\bar \xi}|^2  = \sum_{\alpha = 1}^N\,\int _{R_\alpha} | f^\alpha_{\bar \xi}|^2 \\
& \; =  \frac{1}{2}  \sum_{\alpha = 1}^N\,\int _{R_\alpha} \left (| f^\alpha_{\bar \xi}|^2 \,+\, | f^\alpha_{\xi}|^2\right)\; \;\;\;\big(\textnormal{by formula (\ref{sumOfsquares})}\;\big) \\ & (iii)\quad \quad  \geqslant \sum_{\alpha = 1}^N\,\int _{R_\alpha} | f^\alpha_{\bar \xi} f^\alpha_{\xi}|\\& (iv)\quad  \quad \geqslant \sum_{\alpha = 1}^N\,\left |\,\int _{R_\alpha}  f^\alpha_{\bar \xi} f^\alpha_{\xi}\,\right | \\ & \;(v)\quad\quad \geqslant \left |\, \sum_{\alpha = 1}^N\,\int _{R_\alpha}  f^\alpha_{\bar \xi} f^\alpha_{\xi}\,\right |  \\ & = \left | \,\sum_{\alpha = 1}^N\,\int _{R_\alpha}  \left(\mathcal A_\alpha \,\eta_{\bar \xi} \right)\; \left(\mathcal A_\alpha \eta_\xi \,+\, \mathcal A_\alpha ' \eta\right)\,\right | \\ & = \left | \,\sum_{\alpha = 1}^N\,\int _{R_\alpha}  \left(\mathcal A_\alpha \,\eta_{\bar \xi} \right)\; \left(\mathcal A_\alpha \eta_\xi \,+\, \mathcal A_\alpha ' \eta\right)\,\right | \\ & = \left | \, \sum_{\alpha = 1}^N\,\int _{R_\alpha}  \left(\mathcal A_\alpha \,\eta_{\bar \xi} \right)\; \left(\mathcal A_\alpha \eta_\xi \right)\,\right |\;\;\;\big(\textnormal{due to identity (\ref{AnIdentity})}\;\big) \\ &  =\,  \left |\, \sum_{\alpha = 1}^N\,\int _{R_\alpha}  \mathcal A_\alpha^2 \,\eta_{\bar \xi} \; \eta_\xi\,\right |    = \,  \left |\,\int_\Omega \mathcal H(\xi)  \;\eta_\xi\, \eta_{\bar \xi}\,\right |
\end{split}
\end{equation}
completing the proof of Lemma \ref{InequalityForH}.

\subsection{Proof of Theorem \ref{SecOrdInequality}}

Take a quick look at two simple estimates:
\begin{equation} \label{FirstInequality}
\begin{split}
\frac{1}{2}  \int_\Omega \left(   \abs{h_\xi}^2 + \abs{h_{\bar \xi}}^2       \right) \abs{\eta_{\bar \xi}}^2  \geqslant  \int_\Omega   |h_\xi\,\overline{h_{\bar \xi}}|    \; \abs{\eta_{\bar \xi}}^2 \;=\; \int_\Omega   |\mathcal H(\xi)|    \; \abs{\eta_{\bar \xi}}^2
\end{split}
\end{equation}
and
\begin{equation}\label{SecondInequality}
\begin{split}
\re \int_\Omega h_\xi \overline{h_{\bar\xi}} \;\eta_\xi\, \eta_{\bar \xi}  \;\textbf{d} \xi\; \geqslant   \;-\, \left |\int_\Omega h_\xi \overline{h_{\bar\xi}} \;\eta_\xi\, \eta_{\bar \xi}\right |\; =  -\, \left |\int_\Omega \mathcal H(\xi)  \;\eta_\xi\, \eta_{\bar \xi}  \;\textbf{d} \xi\;\right |
\end{split}
\end{equation}

If we appeal to (\ref{GeneralHolomorphicInequality}) in Lemma \ref{InequalityForH}, then (\ref{QuadraticInequality})  itself follows as a consequence.

\section{Backwards Analysis}\label{Section8}  When reading the above proof backwards,  we recover  precise circumstances under which we have  equality at  (\ref{QuadraticInequality}) of Theorem  \ref{SecOrdInequality}.

For the equality in (\ref{QuadraticInequality})\;it is necessary and sufficient that equality occurs in  (\ref{FirstInequality}), (\ref{SecondInequality}) and  in every link (iii), (iv), (v) of the chain  (\ref{Chain}).
We begin with  (\ref{FirstInequality}),  where the equality occurs if and only if
\begin{equation}\label {(i)}
\abs{h_\xi}^2 \,|\eta_{\bar \xi }\,|^2\,=\, \abs{h_{\bar \xi}}^2 \,|\eta_{\bar \xi }\,|^2\,\;,\;\; \textnormal{almost everywhere in}\,\,\Omega\,.
\end{equation}
Equivalently\;,
\begin{equation} \label{ZeroJacobian}
J_h(\xi)\, \eta_{\bar{\xi}}(\xi) =   0\;\;\;,\;\;\;\textnormal{where}\;\; J_h(\xi)\;=\; \abs{h_\xi(\xi)}^2 \,-\, \abs{h_{\bar \xi}(\xi)}^2 \,
\end{equation}
Thus $\,J_h(\xi) = 0\,$ almost everywhere  in $\,\Omega_\circ \bydef \{ \xi \;;\;  \eta_{\bar{\xi}}(\xi) \not= 0\,\}\,$. By chance, this observation gives the desired Inequality (\ref{Stronginequality}). Precisely, we have

\begin{theorem} \label{StrictInequality} Let $\,\Omega \subset \mathbb C\,$ be any bounded domain and $\,h \in\mathscr W^{1,2}(\Omega)\,$ a Hopf harmonic map whose Jacobian determinant  $\,J_h(\xi) \not = 0 \,$ almost everywhere in $\,\Omega\,$. Then for every test functions  $\,\eta \not\equiv 0\,$, we have strict inequality
$$
\mathcal E[h]  < \mathcal E[H_\varepsilon]\,,
$$
provided $\, \varepsilon\;$ is sufficiently small and different from  $0$ .
\end{theorem}

Now, resuming the backward analysis, we see that equality in (\ref{SecondInequality}) occurs if and only if the following integral is a real nonpositive number,
\begin{equation}\label{(ii)}
\int_\Omega h_\xi \overline{h_{\bar\xi}} \;\eta_\xi\, \eta_{\bar \xi}\;\leqslant 0 \;\;\;\;\textnormal{(a nonpositive real number)}
\end{equation}
Next we take a look at the chain of inequalities in (\ref{Chain}).
For equality in(\ref{Chain}) (iii) it is necessary and sufficient that $| f^\alpha_\xi | \equiv | f^\alpha_{\bar \xi}| $ almost everywhere in  $R_\alpha$ {for all} $\alpha = 1,2, ... ,\,N\,$. This means that  for all \,$\, \alpha = 1,2, ... ,\,N\,$, we should have:
\begin{equation}\label{(iii)}
 f^\alpha_{\bar \xi} \,\equiv  c_\alpha(\xi)\,\overline{ f^\alpha_{\xi}}\,
\end{equation}
where the complex coefficients have constant modulus, $\, | c_\alpha(\xi) | \equiv 1\,$.\\
With these equations in hand, we see that (\ref{Chain}) (iv) becomes an equality if an only if
$$
\int_{R_\alpha}  | f^\alpha_\xi |^2 = \Big |\int_{R_\alpha}  c_\alpha(\xi) \, | f^\alpha_\xi |^2  \Big | \;,\; \textnormal{for every}\; \alpha = 1, 2, ... ,\, N .
$$
 This, in view of (\ref{(iii)}),  is possible if and only if  for all \,$\, \alpha = 1,2, ... ,\,N\,$,
 \begin{equation}\label{(iv)}
  f^\alpha_{\bar \xi} \,\equiv  c_\alpha\,\overline{ f^\alpha_{\xi}}\, ,\,\;\;\textnormal{where the complex coefficients} \, c_\alpha\; \textnormal{are constants.}
 \end{equation}

On the other hand, to have, equality in (\ref{Chain}) (v) it is required that
$$
\sum_{\alpha = 1}^ N  \int_{R_\alpha} |f^\alpha_\xi |^2   \, =\,\Big|  \sum_{\alpha = 1}^ N c_\alpha  \int_{R_\alpha} |f^\alpha_\xi |^2 \, \Big |
$$
This means that $\,c_\alpha \, $ should be the same constants whenever   $\,\int_{R_\alpha} |f^\alpha_\xi |^2 \,\not = 0\,$, $\, \alpha = 1,2, ... ,\,N\,$. \\
 All the above conditions boil down to one equation. Namely, there is a complex constant $\,c\,$ of modulus $\,1\,$ such that $\,f^\alpha_{\bar \xi} \,\equiv  c\,\overline{ f^\alpha_{\xi}}\,,\;\, \textnormal{on every rectangle} \,R_\alpha\,.$ In this way we arrive at the Cauchy-Riemann equations
 $$\,\frac {\partial \xi}{\partial \overline{\xi}} \left[f^\alpha \,- \; c\,\overline{ f^\alpha}\,\right] \,,\;\, \textnormal {  on every rectangle } \,R_\alpha\,.$$

 It is not generally true that the holomorphic functions $\,f^\alpha \,- \; c\,\overline{ f^\alpha}\,$ on $\,R_\alpha\,$ and $\,f^\beta \,- \; c\,\overline{ f^\beta}\,$ on the adjacent rectangle $\,R_\beta\,$  agree  along the common boundary $\,\Gamma _{\alpha \beta} =  \overline{R_\alpha} \cap \overline{R_\beta}\,$. But their squares do agree, so the following function $\,\Psi = \Psi(\xi)\,$\, is holomorphic on the entire domain.

  $$
 \Psi(\xi) =  \mathcal H\,\eta^2 - 2 c\, |\mathcal H |\,|\eta|^2  + c^2\,\overline{\mathcal H} \,\overline{\eta} ^2 \;=\; \left [
 f^\alpha(\xi) -  c\,\overline{ f^\alpha(\xi)}\,\right]^2 \;,\; \textrm{for} \,\xi \in R_\alpha .
 $$
Such a function $\,\Psi\,$, being equal to zero near $\,\partial \Omega\,$, must vanish in the entire domain. This yields

\begin{equation}
 f^\alpha(\xi) \,- \; c\,\overline{ f^\alpha(\xi)} \equiv 0 \;,\;\;\textnormal{on every rectangle}\; R_\alpha \, .
\end{equation}
Since $\,f^\alpha = \mathcal A_\alpha \eta\,$, this reads as $\,\mathcal A_\alpha \eta \,=\, c \overline{\mathcal A_\alpha }\,\overline{\eta }\,$. Multiplying by $\,\mathcal A_\alpha\,$ we arrive at the condition free of the index $\,\alpha \in \{1, 2, ... , N\}\,$; namely, $\,\mathcal H\, \eta \,= c \,|\mathcal H | \,\overline{\eta}\,$. Let us name such $\,\eta \in \mathscr C^\infty_0(\Omega)\,$ a critical direction in the change of the variables.
\begin{theorem} \label{CriticalDirection} Let $\, h \in \mathscr W^{1,2}_{\textnormal{loc}}(\Omega)\,$ be Hopf harmonic and $\,\mathcal H(z) = h_z \overline{h_{\bar z}}\,$. Then we have equality in (1,15)  and in (1,16) if an only if  there is a complex constant $\, c \,$ of modulus 1 such that

\begin{equation}\label{CriticalCircumstances}
\mathcal H\, \eta \,= c \,|\mathcal H | \,\overline{\eta}  \,, \; \, \textnormal {everywhere in } \, \Omega
\end{equation}
\end{theorem}
We leave it to the reader to describe  when such condition  actually occurs.

\section{A Brief Recollection of Quadratic Differentials}
 The reader is referred to \cite{Jenkins3},  \cite{Jenkins4}, \cite{Kuzmina}  for definitions and additional information. There is an interesting abstraction, invented by M. Thurston~\cite{Th} under the name \textit{measured foliations}, of the trajectory structures  and metrics induced by quadratics differentials, see \cite{HubbardMasur}. To a certain extent  the 2-valued mappings in Section \ref{2valeudMappings} are reminiscent of these ideas.
However, our discussion is confined upon results found in the seminal book by K. Strebel  \cite{Strebel}.  Let us extract the following useful facts from this book. \\
\subsection{Simply connected Domains}\label{SimplyConnectedTheorems}
Let us begin with:
\begin{itemize}
\item Theorem 14.2.1 in \cite{Strebel} (page 72) \\
\textit{Let $\,\varphi(z) \, \textnormal dz \otimes \textnormal d z\; \not\equiv 0\,$ be a holomorphic quadratic differential in a simply connected domain $\,\Omega\,$.  Then any two points of $\,\Omega\,$ can be joined by at most one geodesic arc.} \\
In particular, the union of two geodesic arcs cannot contain a closed Jordan curve.
\item Theorem 15.1 (page 74) \\ \textit{Every maximal geodesic arc (in particular every noncritical trajectory) of a holomorphic quadratic differential in a simply connected region is a cross cut.}\\ This means that a noncritical trajectory has two different end-points, both are at the boundary of $\, \Omega\,$. \\
    \item Theorem 16.1 in \cite{Strebel} (page 75)  \\\textit{Let $\,\mathcal H \not \equiv 0\,$ be a holomorphic quadratic differential in a simply connected domain $\,\Omega\,$ and $\,\gamma\,$ its geodesic arc (in particular noncritical trajectory arc) connecting $\,z_\circ\,$ and $\,z_1\,$. Then the $\,\mathcal H\,$-length $\,|\widetilde{\gamma} |_\mathcal H\,$ of any curve $\,\widetilde{\gamma}\not = \gamma\,$ which connects
        $\,z_\circ\,$ and $\,z_1\,$ within $\,\Omega\,$ is larger than } $\,|\gamma |_\mathcal H\,$. \\
\end{itemize}
We recall what this means,
\begin{equation}\label{LengthEstimate}
\begin{split}
&|\widetilde{\gamma} |_\mathcal H \bydef   \int_{\widetilde{\gamma}} \,\sqrt{|\mathcal H(\xi)|}\, \,|\textnormal{d} \xi | \;\;\geqslant  \left |\int_{\widetilde{\gamma}} \,\sqrt{\mathcal H(\xi)}\, \,\textnormal{d} \xi \right| \; \\&  =  \left|\int_\gamma \,\sqrt{\mathcal H(z)}\, \,\textnormal{d} z \right|   =  \int_\gamma \,\sqrt{|\mathcal H(z)}|\, \,|\textnormal{d} z |   \bydef  |\gamma |_\mathcal H
\end{split}
\end{equation}

    As a consequence of the above facts, we see that:
    \begin{theorem}[Partition into Strip Domains] \label{StripPartition}
    Let $\,\varphi(z)\, \textnormal{d}z \otimes \textnormal{d} z\,\not \equiv 0\, $ be a holomorphic quadratic differential defined in a simply connected domain $\,\Omega\,$. Denote by $\,\mathcal C\,\subset \Omega\,$  the union of vertical trajectories passing through the zeros of  $\,\varphi\,$, the so-called critical graph  of $\,\varphi(z)\, \textnormal{d}z \otimes \textnormal{d} z\,$.  Then $\,\Omega\setminus \mathcal C \,$ has full measure in $\,\Omega\,$ which can be decomposed into vertical strips.
    \begin{equation}
    \Omega \setminus \mathcal C  = \bigcup_{\alpha \in \mathbb N}  \Omega_\alpha
    \end{equation}
    \end{theorem}
\begin{definition} Here and in the sequel the term \textit{vertical strip} refers to a simply connected domain swept out by vertical crosscuts of   $\,\varphi(z)\, \textnormal{d}z \otimes \textnormal{d} z\,\not \equiv 0\, $. We emphasize that in our terminology the vertical crosscuts are the noncritical vertical trajectories with two different  endpoints in $\,\partial \Omega\,$.
\end{definition}
\subsection{Multiply Connected Domains}
One of the inherent difficulties to deal with the multiply connected domains is the presence  of recurrent trajectories of a Hopf differential. Actually, it holds that:
\begin{itemize}
\item
 \textit{No trajectory ray of a Hopf differential $\,\mathcal H(z)\,\textnormal d z\otimes \textnormal d z\,$  in a domain of connectivity $\,\leqslant 3\,$ is recurrent.} \\
For a proof  see  J.A. Jenkins \cite{Jenkins1} and \cite {Jenkins2},  and  W. Kaplan \cite{Kaplan} . \\
\item Theorem 17.4 in \cite{Strebel} (page 82)\\
  \textit{Suppose $\,\mathcal H(z)\, \textnormal{d}z \otimes \textnormal{d} z\,\not \equiv 0\, $ is a holomorphic quadratic differential defined in a domain $\,\Omega\,$ and $\,\gamma \subset \Omega\,$ is a closed geodesic of $\,\mathcal H(z)\, \textnormal{d}z \otimes \textnormal{d} z\,$ . Then, any closed curve $\,\widetilde{\gamma} \subset \Omega\,$ in the homotopy class of\, $\,\gamma\,$ has $\,\mathcal H\,$-length $\,|\widetilde{\gamma }|_{\mathcal H} \geqslant \,|\gamma |_{\mathcal H}\,$.}
\end{itemize}

\begin{figure}[!h]
\begin{center}
\includegraphics*[height=2.0in]{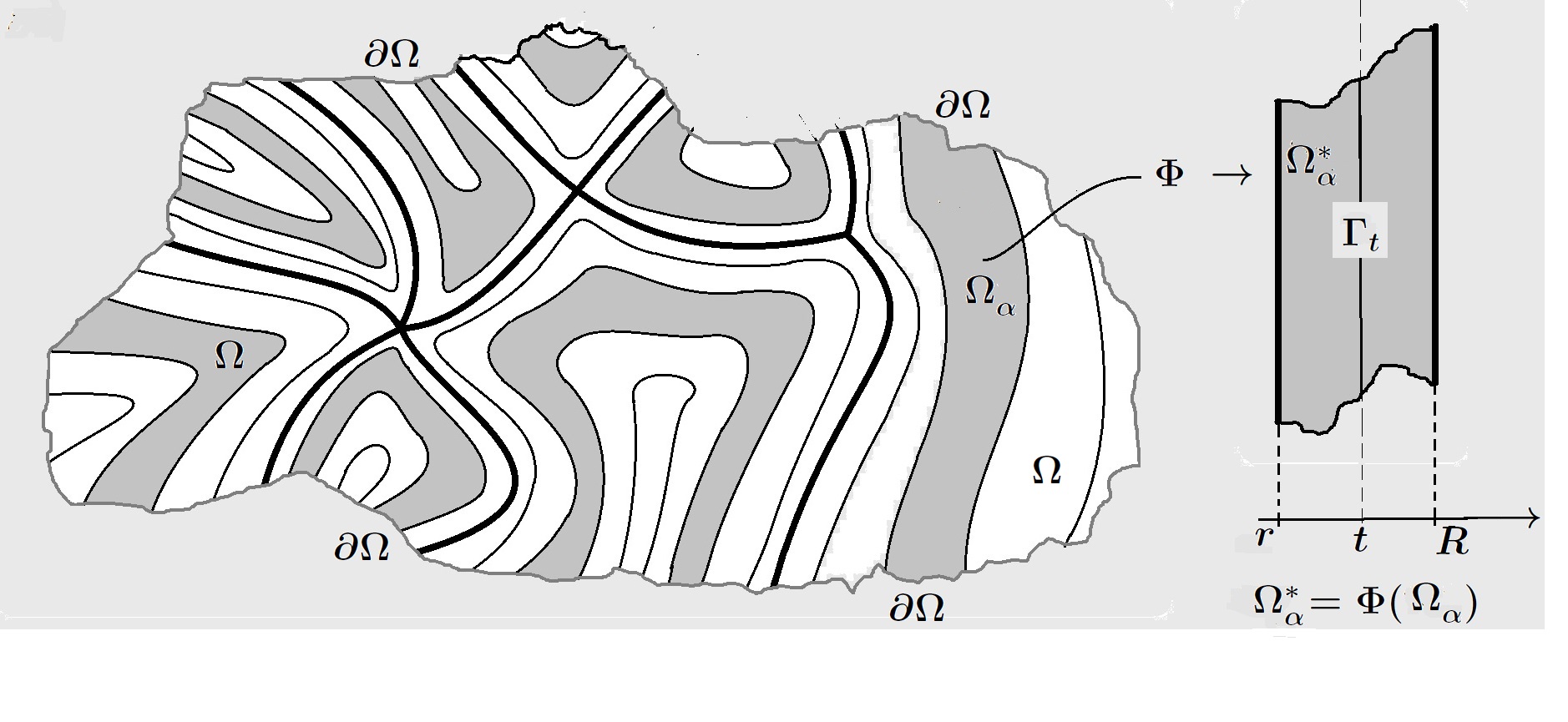}
\caption{A strip type domain $\,\Omega_\alpha\,$ is swept out by vertical trajectory arcs with endpoints at $\,\partial \Omega\,$. The conformal transformation  $\,\Phi = \Phi(z) \bydef \int \sqrt{\phi(z)} \textnormal d z\,$ (so-called distinguished parameter) takes those arcs into vertical straight line segments $\,\Gamma_t\,\,, r< t < R\,$, which form a Euclidean strip $\,\Omega_\alpha ^*\,$. \,}\label{Figure2}
\end{center}
\end{figure}

From now on,  we make a standing assumption  that $\,\mathcal H(z) \,\textnormal d z \otimes \textnormal d z \,$ admits only two types of configuration domains (possibly a countable number of them) ; namely,
\begin{itemize}
\item  The \textit{strip domains} and
\item The \textit{circular domains}; each of which is swept out by closed vertical trajectories.
\end{itemize}
Precisely, we have a disjoint union of full area in $\,\Omega\,$
 \begin{equation}\label{ConfigurationDomains}
 \Omega ^{'} \; \bydef\; \bigcup_{\alpha \in \mathbb N}  \Omega_\alpha\; \subset \Omega,\;\; |\Omega \setminus \Omega^{'} |  =  0
 \end{equation}
where $\,\Omega_\alpha\,$ is either a circular domain or a strip domain. Such configurations typically occur upon restriction to $\Omega\,$ of a  Strebel quadratic differential on the Riemann sphere $\,\widehat{\mathbb C}\,$ (that is, having only closed trajectories). In this case the vertical crosscuts are  non other than the fragments of closed trajectories that lay within $\Omega\,$,  see Figures  \ref{FigureLT},  \ref{FourPoles}.  We refer to such  $\,\mathcal H(z) \,\textnormal d z \otimes \textnormal d z \,$ as \textit{Strebel type differential on $\,\Omega\,$}.

\section {The Length-Area Inequalities}\label{LengthArea}
We note that for $\,h \in \mathscr W^{1,2}(\Omega)\,$ the differential  $\,\mathcal H(z) \,\textnormal d z \otimes \textnormal d z \,$ has finite area; meaning that $\,\int_\Omega  |\mathcal H(z) |  \,\textbf{d} z  < \infty\,$.
\begin{proposition} [Length-Area Inequalities]\label{LengthAreaIneq}
Let $\,\mathcal H(z)\, \textnormal{d}z \otimes \textnormal{d} z\,\not \equiv 0\, $  be a Strebel differential in $\,\Omega\,$ of finite area,  and let $\,F\,$ and $\,G\,$ be measurable functions in $\,\Omega\,$   such that
$$
\int_\Omega |F(z)|\, |\mathcal H(z) |\,\textnormal{\textbf{d}} z\; < \infty\;\;\; \textnormal{and}\;\;\;\int_\Omega |G(z)|\, |\mathcal H(z) |\,\textnormal{\textbf{d}} z\; < \infty
$$
Suppose that for every vertical trajectory $\,\gamma \subset \Omega \,$ (either circular or crosscut, see formula \ref{ConfigurationDomains}) the following inequality holds:
\begin{equation}\label{LineInequalities}
\int_\gamma |F(z)|\, \sqrt{|\mathcal H(z) |}\,|\textnormal{d} z|\; \leqslant \;\int_\gamma |G(z)|\, \sqrt{|\mathcal H(z) |}\,|\textnormal{d} z |\;
\end{equation}
Then
\begin{equation}\label{AreaInequalities}
\int_\Omega |F(z)|\, |\mathcal H(z) |\,\textnormal{\textbf{d}} z\; \leqslant  \;\;\int_\Omega |G(z)|\, |\mathcal H(z) |\,\textnormal{\textbf{d}} z\;
\end{equation}
\end{proposition}
\begin{remark}
This Proposition  reduces to Fubini's Theorem upon a conformal change of variables in both the line and the area integrals.
\end{remark}
\begin{proof} Since the set $\,\Omega \subset \Omega\,$ is a disjoint union of configuration domains in  which the line inequalities (\ref{LineInequalities})  hold, the problem reduces equivalently to showing that
\begin{equation}\label{AreaWithAlpha}
\int_{\Omega_\alpha} |F(z)|\, |\mathcal H(z) |\,\textnormal{\textbf{d}} z\; \leqslant  \;\;\int_{\Omega_\alpha} |G(z)|\, |\mathcal H(z) |\,\textnormal{\textbf{d}} z\;,\;\;\textnormal{for every}\; \alpha .
\end{equation}
\textbf{Case 1.} \textit{$\,\Omega_\alpha\,$ is a strip domain}. The so-called distinguished parameter $\,\Phi(z) \bydef \int  \sqrt{\mathcal H(z)} \,\textnormal d z\,$ defines a conformal transformation of $\,\Omega_\alpha\,$ onto Euclidean strip $\, \Phi(\Omega_\alpha)\,\subset \mathbb C\,$ which is swept out by straight line vertical segments, say $\,\Gamma_t = \{ w \in \Phi(\Omega_\alpha) \,;\; \Re e\, w  =  t\,\}\,$ for $\,r < t < R\,$, see Figure \ref{Figure2}.

The  area element $\, \textbf{d} z\,$ upon  the transformation  $\, z = \Phi^{-1}(w)\,$, reads as $\, \textnormal d w = \; |\Phi'(z) |^2 \textnormal d z\,$ where  $\, w \in \Phi(\Omega_\alpha)\,$
 Accordingly, we have

\begin{equation}\label{ChainOfEqualities}
\begin{split} & \quad \quad \quad \quad\int_{\Omega_\alpha} |F(z)|\, |\mathcal H(z) |\,\textnormal{\textbf{d}} z\; = \int_{\Omega_\alpha} |F(z)|\, |\Phi'(z) |^2\,\textnormal{\textbf{d}} z\; \;\; = \\& = \;\int_{\Phi(\Omega_\alpha)} | F(\Phi^{-1}(w) ) | \,\textnormal{\textbf{d}} w\,  \;=\; \int_r^R \left( \int_{\Gamma_t} | F(\Phi^{-1}(w) ) | \,|\textnormal{d} w| \right)\; \textnormal d t   = \\& = \int_r^R \left(\int_{\gamma_t}  | F(z) |  \sqrt{| \Phi'(z) |} \, |\textnormal dz|  \right)  \textnormal d t   = \int_r^R \left(\int_{\gamma_t}  | F(z) |  \sqrt{| \mathcal H(z) |} \, |\textnormal dz|  \right)  \textnormal d t
\end{split}
\end{equation}
Here $\,\gamma_t\,\bydef \Phi^{-1}(\Gamma_t) \subset \Omega_\alpha\,$ is one of the vertical trajectory arcs in $\,\Omega_\alpha\,$ with endpoints in $\,\partial \Omega\,$. By virtue of (\ref{LineInequalities}), if one replaces $\,F\,$ with $\,G\,$ in the  line integral over $\,\gamma_t\,$ it will increase the integral. Then, upon such replacement, we return to the area integral for $\,G\,$ by reversing the sequence of the  identities in (\ref{ChainOfEqualities}). This results in the desired inequality (\ref{AreaWithAlpha}).\\

\textbf{Case 2.} \textit{$\,\Omega_\alpha\,$ is a circular domain}. The proof goes through in much the same way as for the strip domains. In this case, however,  $\,\Omega_\alpha\,$ is swept out by closed vertical trajectories $\,\gamma \subset \Omega_\alpha\,$. They have  the same $\,\mathcal H\,$-length  $\,\ell \bydef \int_\gamma \sqrt{|\mathcal H(z)|}\;| \textnormal d z |\, = \pm i \int_\gamma \sqrt{\mathcal H(z)}\; \textnormal d z \, $. Here we choose a continuous branch of $\,\sqrt{\mathcal H(z)}\,$ in $\,\Omega_\alpha \setminus \mathcal C\,$, where $\,\mathcal C \,$ is a horizontal cut of $\,\Omega_\alpha\,$. This gives us  a conformal transformation
$$  \Phi(z)  \bydef \exp \left(\frac{2\pi}{\ell} \int \sqrt{\mathcal H(z)}\,  \textnormal d z\right) \,$$
of $\,\Omega_\alpha\,$ onto an annulus swept out by concentric circles, say
$$ \Phi(\Omega_\alpha)   =  \bigcup_{r<t<R} \Gamma_t\;,\;\textnormal{where}\; \Gamma_t  = \{ w \in \mathbb C\,;\, |w| = t\; \}  $$
The rest of the proof runs as in (\ref{ChainOfEqualities}) with hardly any changes.
\end{proof}

\section{Proof of Theorem  \ref{FullDirichletPrinciple}}
We follow analysis similar to that in \cite{MonHopfHar}.
Let $\,h : \Omega \rightarrow \mathbb C \,$ be a mapping of Sobolev class  $\,\mathscr W^{1,2} (\Omega)\,$. For the moment  both $\,\Omega\,$ and $\,h\,$ are arbitrary, to be specified later. Consider a diffeomorphism $  \,f : \Omega \onto \Omega\,$ and the corresponding inner variation of $\,h\,$  defined by the rule
$$ H(w)  =  h(f^{-1}(w)\, \textnormal {; equivalently,} \; H(w)  =  h(z) \;,\textnormal{for} \; w = f(z)\;$$

\begin{lemma}\label{MainInequality} We have the following inequality:
\begin{equation}
\begin{split} \mathscr E[H]  - \mathscr E[h]  & \;\; = \\& = \; 2\int_\Omega \left [ \frac{| f_z \,-\, \frac{\varphi}{|\varphi|}\, f_{\bar z}\; |^2}{| f_z|^2 \;-\; |f_{\bar z} |^2}   \; -\; 1 \right ] \,| \varphi | \, \textbf d z \,+ \\& + 2\int_\Omega  \frac{\left (| h_z |\,-\, |\, h_{\bar z}\; |\right )^2\; |f_{\bar z} |^2}{| f_z|^2 \;-\; |f_{\bar z} |^2}  \, \, \textbf d z \,+ \\& \geqslant \; 2\int_\Omega \left [ \frac{| f_z \,-\, \frac{\varphi}{|\varphi|}\, f_{\bar z}\; |^2}{| f_z|^2 \;-\; |f_{\bar z} |^2}   \; -\; 1 \right ] \,| \varphi | \, \textbf d z
\end{split}
\end{equation}
where $\,  \varphi \bydef h_z \overline{h_{\bar z}}   \in \mathscr L^1(\Omega) \,$ is a Hopf type product (not necessarily holomorphic).
By a convention,   $\, \frac{\varphi}{|\varphi|} \bydef 0\,$ at the points where $\,\varphi = 0\,$.
\end{lemma}
\begin{proof}
We begin with the inverse map $\,f^{-1} : \Omega \onto \Omega\,$ and its complex derivatives.
$$ \,\frac{\partial f^{-1}(w)}{\partial w}  = \frac{\overline{f_z(z)}}{J_f(z)}\;\;  ,\;\; \frac{\partial f^{-1}(w)}{\partial \overline{w}}  = \frac{ -\; f_{\overline{z}}(z)}{J_f(z)}\;,\;\textnormal{where}\;            z = f^{-1}(w)\,$$
  Using chain rule we obtain
\begin{equation}
\begin{split} &H_w(w)  =  h_z(z) \,\frac{\partial f^{-1}(w)}{\partial w}\; +\; h_{\bar z}(z) \,\frac{\overline{\partial f^{-1}(w)}}{\partial \bar w}\; =\; \frac{h_z\,\overline{f_z}\,-\,h_{\bar z}\,\overline{f_{\bar z}}  }{J_f(z)}\\& H_{\bar w} (w)  =  h_z(z) \,\frac{\partial f^{-1}(w)}{\partial \bar w}\; +\; h_{\bar z}(z) \,\frac{\overline{\partial f^{-1}(w)}}{\partial  w}\,=\; \frac{h_{\bar z}\,\overline{f_z}\,-\,h_z\,f_{\bar z}  }{J_f(z)}
\end{split}
\end{equation}
Hence
\begin{equation}
\begin{split} & \mathscr E[H] \bydef  \int_\Omega \left (\, |H_w|^2  +  |H_{\bar w} |^2  \,\right )\, \textbf d w \, = \\& =
 \int_\Omega \frac{|\,h_z\,\overline{f_z}\,-\,h_{\bar z}\,\overline{f_{\bar z}}\,|^2 \; + \; |\,h_{\bar z}\,\overline{f_z}\,-\,h_z\,f_{\bar z}  \,|^2}{J_f(z) } \textbf d z
\end{split}
\end{equation}
Here we have made a substitution $ w = f(z)\,$, so $\,\textbf d w  = J_f(z) \, \textbf d z\,$. Recall the energy formula for $\,h\,$,  $\,\mathscr E[h]  =    \int _\Omega \left (  |h_z |^2 \; + |\, h_{\bar z} |^2 \right )\,\textbf d z  \,$. Therefore,

\begin{equation}
\begin{split}
&\quad\quad\quad\quad \mathscr E[H] \;-\mathscr E[h] \; =\\&
 \int_\Omega \frac{|\,h_z\,\overline{f_z}\,-\,h_{\bar z}\,\overline{f_{\bar z}}\,|^2 \; + \; |\,h_{\bar z}\,\overline{f_z}\,-\,h_z\,f_{\bar z}  \,|^2\,- \left (  |h_z |^2 \; + |\, h_{\bar z} |^2 \right )\, \left (  |f_z |^2 \; -\, |\, f_{\bar z} |^2 \right ) }{J_f(z) } \textbf d z
\end{split}
\end{equation}
We leave it to the reader a routine computation that leads to the desired formula

\begin{equation}
\begin{split} \mathscr E[H]  - \mathscr E[h]  & \;\; = \\& = \;2 \int_\Omega \left [ \frac{| f_z \,-\, \frac{\varphi}{|\varphi|}\, f_{\bar z}\; |^2}{| f_z|^2 \;-\; |f_{\bar z} |^2}   \; -\; 1 \right ] \,| \varphi | \, \textbf d z \,+ \\& +2 \int_\Omega  \frac{\left (| h_z |\,-\, |\, h_{\bar z}\; |\right )^2\; |f_{\bar z} |^2}{| f_z|^2 \;-\; |f_{\bar z} |^2}  \, \, \textbf d z
\end{split}
\end{equation}
Hence
\begin{equation}\label{EHminusEh}
\mathscr E[H]  - \mathscr E[h]  \; \geqslant \; \;2 \int_\Omega\, \frac{|\, f_z \,-\, \frac{\varphi}{|\varphi|}\, f_{\bar z}\; |^2}{| f_z|^2 \;-\; |f_{\bar z} |^2}  \, \,| \varphi(z) | \, \textbf d z   \;-\; 2 \int_\Omega  \,| \varphi(z) | \, \textbf d z
\end{equation}
This ends the proof of Lemma \ref{MainInequality}.
\end{proof}

Next, using H\"{o}lder's inequality,  we  estimate the first integral in the right hand side of (\ref{EHminusEh}).

\begin{equation} \label{LowerBound}
\begin{split}
&\int_\Omega\, \frac{|\, f_z \,-\, \frac{\varphi}{|\varphi|}\, f_{\bar z}\; |^2}{| f_z|^2 \;-\; |f_{\bar z} |^2}  \, \,| \varphi(z) | \, \textbf d z  \geqslant  \\& \left[\int_\Omega | \varphi ( w) | \,\textbf d w \right]^{-1}\left(\int_\Omega |\, f_z \,-\, \frac{\varphi}{|\varphi|}\, f_{\bar z}\; |\;  \sqrt{|\varphi(z)| } \; \sqrt{|\varphi (f(z)) |}\; \textbf d z\right)^2
\end{split}
\end{equation}

Here is a simple direct computation for this.

\begin{equation}
\begin{split}
& \int_\Omega |\, f_z \,-\, \frac{\varphi}{|\varphi|}\, f_{\bar z}\; |\;  \sqrt{|\varphi(z)| } \; \sqrt{|\varphi (f(z)) |}\; \textbf d z  =\\ &
 =\int_\Omega \frac{|\, f_z \,-\, \frac{\varphi}{|\varphi|}\, f_{\bar z}\; |\;  \sqrt{|\varphi(z)| }}{\sqrt{J_f(z)} } \; \sqrt{J_f(z)\, |\varphi (f(z)) |}\; \textbf d z \\ &
 \leqslant \left [\int_\Omega \frac{|\, f_z \,-\, \frac{\varphi}{|\varphi|}\, f_{\bar z}\; |^2\;  |\varphi(z)| }{J_f(z) } \;  \textbf d z\;\right]^\frac{1}{2}\;\; \left[ \int_\Omega J_f(z)\, |\varphi (f(z)) |\; \textbf d z\,\right]^\frac{1}{2}\\ & =
\left[\int_\Omega\, \frac{|\, f_z \,-\, \frac{\varphi}{|\varphi|}\, f_{\bar z}\; |^2}{| f_z|^2 \;-\; |f_{\bar z} |^2}  \, \,| \varphi(z) | \, \textbf d z \right]^\frac{1}{2}\;\; \left[\int_\Omega | \varphi ( w) | \,\textbf d w \right]^\frac{1}{2} \end{split}
\end{equation}
Whence the estimate (\ref{LowerBound}) is readily inferred.\\

  For the proof of Theorem \ref{FullDirichletPrinciple} we need the following identity.
\begin{lemma} \label{ChangeInLineIntegral}
Let $\,\varphi = \mathcal H \not \equiv 0\,$ be any holomorphic function and $\, \gamma \subset \Omega\,$ a vertical trajectory arc of $\,\mathcal H(z)\, \textnormal d z \otimes \textnormal d z\,$ . Then for every diffeomorphism $\, f : \Omega \onto \Omega\,$ it holds that
\begin{equation}\label{ChangOfVariables1}
\int_\gamma |\, f_z \,-\, \frac{\varphi}{|\varphi|}\, f_{\bar z}\; | \; \sqrt{|\varphi (f(z)) |}\; |\textnormal d z| = \int_{f(\gamma)}\sqrt{ |\varphi(w) |}\, |\textnormal d w | \bydef  | f(\gamma) |_\varphi
\end{equation}
\end{lemma}
\begin{proof}
We use the arc-length parametrization of $\,\gamma\,$, $\, \gamma = \{ z(t) ; \,a < t < b ,\; |{\dot z}(t)|  \equiv 1\,\}\,$. Upon a substitution  $\,w = f(z)\,$ in the line integral over $\,f(\gamma)\,$ , we obtain,
$$
\int_{f(\gamma)}\sqrt{ |\varphi(w) |}\, |\textnormal d w |  =   \int_\gamma  \; \sqrt{|\varphi (f(z)) |}\; |\textnormal d f(z)|  =  \int_\gamma  \; \sqrt{|\varphi (f(z)) |}\; | f_z \textnormal dz  +  f_{\bar z} \textnormal d{\bar z}\,|
$$
Recall the relation $\,\textnormal d{\bar z} = - \frac{\varphi(z)}{| \varphi(z) |}\,\textnormal d{z}\,$  along any vertical trajectory, in which $\, \varphi(z(t))\, [\dot{z}(t) ]^2  < 0\,$. This results in (\ref{ChangOfVariables1}), completing the proof of Lemma \ref{ChangeInLineIntegral}.
\end{proof}

From now on $\,\gamma\,$ will  any noncritical trajectory of $\,\phi(z)\,\textnormal dz \otimes \textnormal{d} z\,,\; \phi = \mathcal H\,$. We shall appeal to  the theorems   listed in Section \ref{SimplyConnectedTheorems}. \\
First, Theorem 15.1 in \cite{Strebel} (page 74) tells us that $\,\gamma\,$ is a cross-cut connecting two different boundary points. Since the diffeomorphism $\, f : \Omega \onto \Omega\,$  equals the identity map near $\,\partial \Omega\,$ the arcs $\,\gamma\,$ and $\,f(\gamma)\,$ coincides near the boundary. By Theorem 16.1 in \cite{Strebel} (page 75), the $\mathcal H\,$-length of $\,f(\gamma)\,$ is larger (or equal)  than the $\mathcal H\,$-length of $\,\gamma\,$. In symbols,
\begin{equation}\label{LengthInequality1}
| f(\gamma)|_{_\phi} \; \geqslant | \gamma|_{_\phi}  =  \;\int _\gamma \sqrt{|\phi(z)|} \, | \textnormal d z|
\end{equation}
Now Lemma \ref{ChangeInLineIntegral} gives the inequality \\
$$
\;\int _\gamma \sqrt{|\phi(z)|} \, | \textnormal d z|  \leqslant \int_\gamma |\, f_z \,-\, \frac{\varphi}{|\varphi|}\, f_{\bar z}\; | \; \sqrt{|\varphi (f(z)) |}\; |\textnormal d z|
$$
Next the length-area inequalities (\ref{LineInequalities}) and (\ref{AreaInequalities})   combined give
$$
 \int_\Omega |F(z)|\, |\phi(z) |\, \textbf{d} z \leqslant \int_\Omega |G(z)|\, |\phi(z) | \, \textbf{d} z
$$
where $\,F(z) \equiv 1\,$ and
$$
 G(z) =  |\, f_z \,-\, \frac{\varphi}{|\varphi|}\, f_{\bar z}\; | \; \sqrt{|\varphi (f(z)) |}\,\; \text{\huge$/$}\;\, \sqrt{|\phi(z)|}
$$
This reads as:
\begin{equation}\label{AreaInequality}
\int_\Omega |\, f_z \,-\, \frac{\varphi}{|\varphi|}\, f_{\bar z}\; |\;  \sqrt{|\varphi(z)| } \; \sqrt{|\varphi (f(z)) |}\; \textbf d z\;\geqslant \; \int_\Omega | \varphi ( w) | \,\textbf d w
\end{equation}
Substituting this into (\ref{LowerBound}), in view of ({\ref{EHminusEh}}),  we conclude with  the desired inequality $\,\mathscr E[H]  -  \mathscr E[h]  \geqslant  0\,$.

\section{Dirichlet Principle in Multiply Connected Domains}\label{sec:mcd}

\begin{theorem}[Dirichlet Principle for Multiply Connected Domains] \label{MultiplyConnectedPrinciple}
 Suppose that a Hopf holomorphic differential $\,h_z \overline{h_{\bar z}} \,$  for $\,h \in \mathscr W^{1,2} (\Omega)\,$ is of a Strebel type.  Then every nontrivial inner variation  of $\,h\,$ increases its energy.
 \end{theorem}
\begin{proof}
 The arguments are essentially the same as presented in the proof of Theorem \ref{FullDirichletPrinciple}. The estimates over a strip domains $\Omega_\alpha\,$ are exactly the same. If, however,  $\,\Omega_\alpha\,$ is a circular domain and $\,\gamma \subset \Omega_\alpha\,$ is a closed trajectory,  we still have the desired  length inequality (\ref{LengthInequality1}). The rest of the proof runs in  the same way.
\end{proof}
\subsection{Illustrations of Theorem \ref{MultiplyConnectedPrinciple}  by hyperelliptic trajectories}
The term hyperelliptic quadratic diffferential refers to a meromorphic quadratic differential on the Riemann sphere $\,\widehat{\mathbb C}\,$, see \cite{JenkinsSpencer}.
\subsubsection{Leminiscate}
Consider a quadratic differential
\begin{equation}
\mathcal H(z) \,\textnormal d z \otimes \textnormal d z\; =\; \left ( \frac{z}{1 - z^2} \right )^2 \textnormal d z \otimes \textnormal d z \;,\;\; z \not =\pm 1
\end{equation}
Thus $\,\mathcal H\,$ has one critical point (double zero at $\,z = 0\,$) and two double poles at $\, z = \pm 1\,$.  To every parameter $\, 0 < r \leqslant 1\,$ there corresponds a closed vertical trajectory around the pole at $\,+1\,$.
 \begin{equation}
 z(t)  =  \sqrt{1 \, +\, r^2 e^{4it}}\;,\;\;\; \textnormal{where}\;\, -\frac{\pi}{4} \leqslant  t \leqslant \frac{\pi}{4}
 \end{equation}
 Here the continuous branch of the square root is chosen to make $\, z(0) =  \sqrt{1 + r^2} \,$. Indeed, we have
 $$
 \mathcal H(z(t))\,[ \dot{z}(t) ]^2 =  \frac{1}{4} \left[ \frac{\textnormal d z^2(t) / \textnormal d t}{1 - z^2(t)} \right]^2\; \equiv\,- 4  < 0
 $$
 The borderline case $\, r = 1\,$ results in a closed geodesic curve passing through the critical point $\, z(\frac{\pi}{4} )  \,=\, z(\frac{- \pi}{4}\,) =\, 0\,$. In fact this is the right-half portion of a leminiscate,  $\,z(t) = \sqrt{2 \cos 2t} \; e^{it}\,$, see Figure \ref{FigureLT}.  Changing the sign of the square root gives us closed trajectories around the pole at $\-1\,$. In particular, the borderline case $\,r = 1\,$ results in the left-half portion of the leminiscate.\\
 There are also closed trajectories surrounding both poles. To every $\; R > 1\,$ there corresponds a closed trajectory:
 \begin{equation}
 z(t)\,=\, F(R e^{it} )\;, \, 0 \leqslant t \leqslant 2 \pi, \;\, \,\textnormal{where}\;  F(\xi) = \xi \sqrt{ 1 + \xi^{-2}}\,, \; |\xi| > 1
 \end{equation}
 The continuous branch of square root is chosen to make $\,F(1) =  \sqrt{2}\,$.\\

 Let us restrict  $\,\mathcal H(z) \,\textnormal d z \otimes \textnormal d z\,$  to a bounded domain $\,\Omega\,$ which contains no poles,    $\, \pm 1 \not \in \overline{\Omega}\,$.   Every vertical noncritical trajectory in $\,\Omega\,$  is either a closed Jordan curve or its intersection with $\,\Omega\,$. The latter consists of a number (possibly countable) of cross-cuts. In Figure \ref{FigureLT} the shaded area occupies the domain $\,\Omega\,$ of connectivity 4. Two darker fragments represent ring and strip regions.  Every closed curve $\,\widetilde{\gamma} \subset \Omega\,$ that is homotopic to a closed trajectory $\,\gamma \subset \Omega\,$ around the double pole at 1 has $\,\mathcal H\,$-length $\, |\,\widetilde{\gamma}|_\mathcal H  \geqslant |\gamma\,|_\mathcal H = \pi\,.$

 \begin{figure}[!h]
\begin{center}
\includegraphics*[height=2.0in]{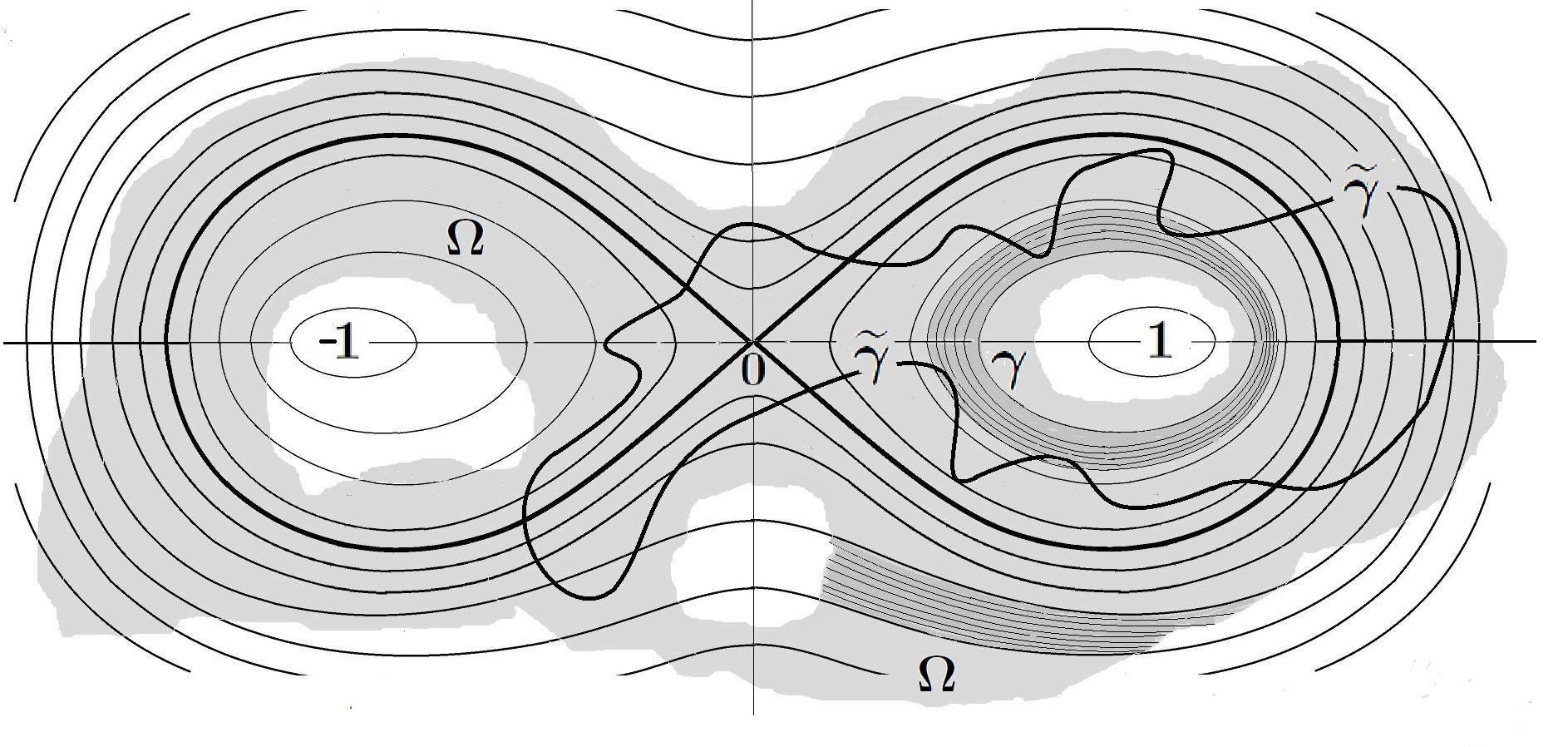}
\caption{Leminiscate as a critical graph \,}\label{FigureLT}
\end{center}
\end{figure}

\subsubsection{Leminiscates with four poles}
Here is another example of a rational quadratic differential with leminiscates as trajectories, see Figure \ref{FourPoles}.
\begin{equation}
 \left [\frac{5}{z-2} \;+\;\frac{5}{z+2}\;+\; \frac{7}{z-4}\;+\; \frac{7}{z+4}\right ]^2\; \textnormal d z \otimes \textnormal d z
\end{equation}

\begin{figure}[!h]
\begin{center}
\includegraphics*[height=2.1in]{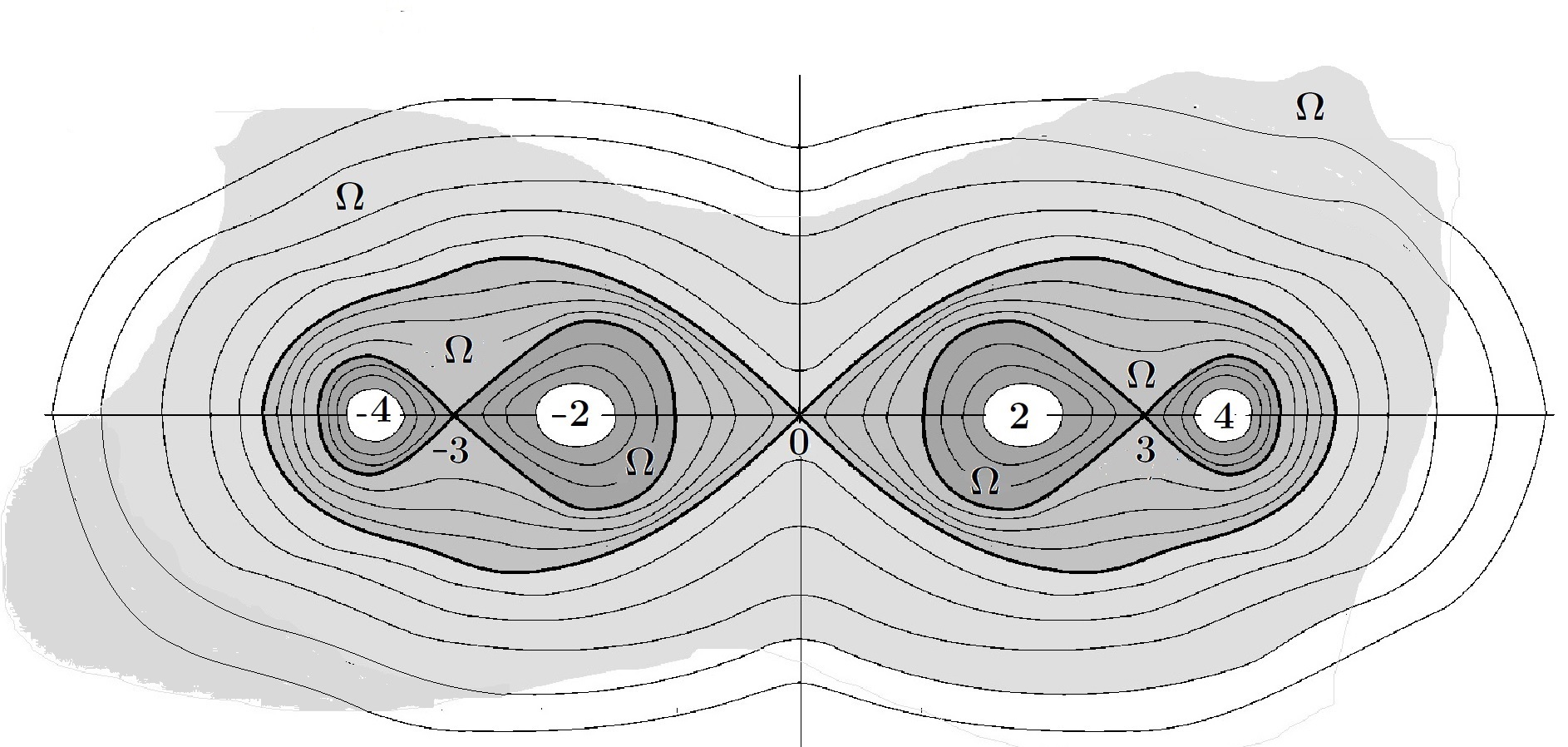}
\caption{Meromorphic  differential with three critical points at $\,0\,$  and $\,\pm 3\,$,  and with four double poles at $\,\pm 2\,$ and $\, \pm 4\,$ .}\label{FourPoles}
\end{center}
\end{figure}
\begin{remark} In the above examples of leminiscates the meromorphic function $\,\mathcal H\,$ admits a continuous square root on $\,\widehat{\mathbb C}\,$. In this case there is a simple direct proof of the minimal length property of closed trajectories  as stated in Theorem 17.4 in \cite{Strebel} (page 52). The proof goes through  as for (\ref{LengthEstimate})\, in two lines with hardly any changes.
\end{remark}

\subsubsection{A hyperelliptic differential having no square root}

Consider the polynomial with roots $\; a_k = \exp \frac{(2k+1)\pi i}{n}\,,\; k = 0, 1, ... , n-1$
$$
z^n + 1 = (z - a_0) (z -a_1) \cdots  (z-a_{n-1} )\;,
$$
Upon differentiation we get the formula,
$$
\frac{n z^{n-1}}{z^n + 1}\; =\; \frac{1}{z-a_0}\;+\;\frac{1}{z - a_1}\;+ \cdots + \frac{1}{z - a_{n-1}}
$$
Second differentiation yields,
$$
\frac{n z^{n-2} \left(z^n -n + 1   \right)}{(z^n + 1 )^2} \; =\; \frac{1}{(z-a_0)^2}\;+\;\frac{1}{(z - a_1)^2}\;+ \cdots + \frac{1}{(z - a_{n-1})^2} \; \bydef \mathcal H(z)
$$
whence  $\,\mathcal H(z) \,\textnormal d z \otimes \textnormal d z \,$ has a critical point of order $\,n-2\,$ at $\,z = 0\,$.  Moreover,  it has $\,n\,$ critical points of order 1 at  $\; z_k = \sqrt[n]{n-1}\,\exp \frac{2 k \pi i}{n}\,,\; k = 0, 1, ... , n-1\,$, see Figure \ref{OrangeSections}.
\begin{figure}[!h]
\begin{center}
\includegraphics*[height=2.1in]{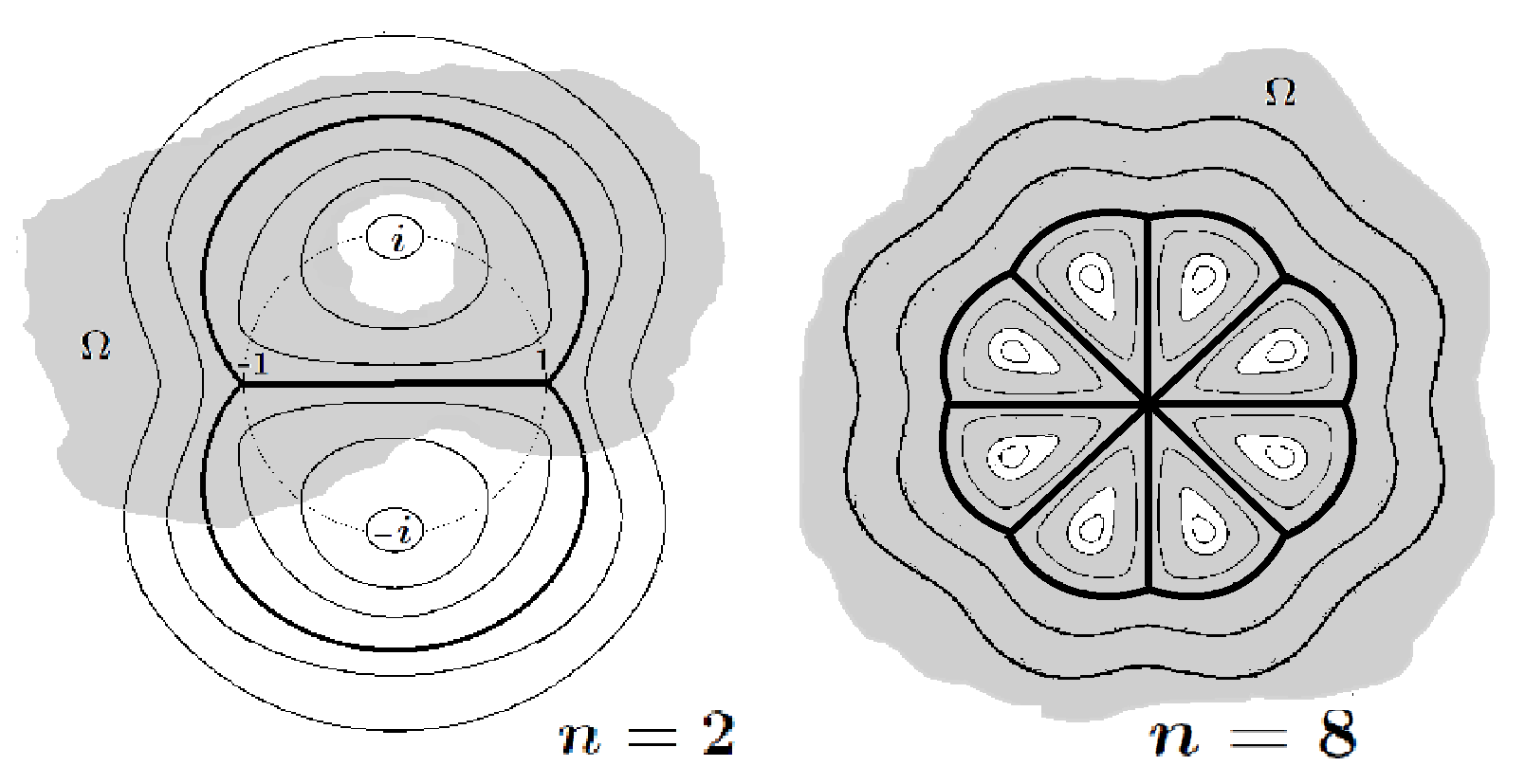}
\caption{Hyperelliptic trajectories with $\,n\,$ double poles and $\,2n-2\,$ critical points counting multiplicity.\label{OrangeSections}}
\end{center}
\end{figure}
Our interest in this example comes from  \cite{IwaniecMichigan}, where certain sharp estimates for hyperelliptic differentials have been established in connection with the area distortion inequality for quasiconformal mappings.

\end{document}